\documentclass[final]{siamltex}
\usepackage{amsfonts}
\usepackage{amssymb}       
\usepackage{amsmath}
\usepackage{latexsym}
\usepackage{verbatim}      %
\usepackage{bm,enumerate}
\usepackage{multirow}
\usepackage{array}
\usepackage[dvips]{graphicx}
\usepackage{color}
\usepackage{showkeys}                %
\def\SMALLO{\mathcal{o}}
\def\DELTAT{{\delta t}}
\def\NUMSP{{K}}                       

\def\DERIV{\delta}
\def\DX{\DERIV_x}
\def\DXOMEGA{\DERIV_{x,\omega}}

\def\DY{\DERIV_y}
\def\DZ{\DERIV_z}
\def\DXY{\DERIV_{xy}}
\def\DYX{\DERIV_{yx}}
\def\DXZ{\DERIV_{xz}}
\def\DYZ{\DERIV_{yz}}
\def\DXYZ{\DERIV_{xyz}}
\def\DXK#1{\DERIV_{x_{#1}}}

\def\DIST{\mathrm{dist}\,}

\def\INDXE{\mathcal{I}^1}
\def\INDXO{\mathcal{I}^2}

\def\BNDRY{\partial}
\def\CLSR{\bar}

\def\DIAM{\mathrm{diam}\,}

\def\RANG_S{S}                            

\def\LATT{{\Lambda}_N}                    
\def\SIGMA{{\mathcal{S}}}               

\def\LATTA{\LATT^1}
\def\LATTB{\LATT^2}

\def\CSET{C}

\def\CB{C_b}                     
\def\CBS{\CB({\SIGMA})}
\def\CS#1{C^{#1}}                
\def\CSS#1{\CS{#1}({\SIGMA})}

\def\NORMB#1{\Vert\,#1\,\Vert_\infty}

\def\NORM#1{\NORMB{#1}}
\def\NORMINF#1{\|{#1}\|_\infty}
\def\NORMA#1#2{\|{#2}\|_{#1}}

\def\ISINGOP#1#2{\sum_{x\in #1}c(x,\sigma)[#2(\sigma^x)-#2(\sigma)]}

\def\RATEC{c(x,\omega;\sigma)}    

\def\SEMIGRPt{\EXP{t\LOPER}} 
\def\SEMIGRP{P}                             
\def\SEMIGRPL{\SEMIGRP_{L}}
\def\SEMIGRPS{\SEMIGRP_{S}}
\def\SEMIGRPB{\tilde\SEMIGRP}
\def\SEMIGRPLT{\SEMIGRP_{L}(t)}
\def\SEMIGRPST{\SEMIGRP_{S}(t)}

\def\omegax{{x, \omega}}

\def\SPINSP{\Sigma}                     
                   
\def\CONFNEW{\sigma^{x,\omega}}
\def\CONFNEWX{\sigma^{x,\omega}}

\def\PROCAPPRGAM{\{\gamma_{kh}\}_{k=0}^{n}}

\def\PROCMIC{\{\sigma_t\}_{t\geq 0}}

\def\PROC#1{\{{#1}\}_{t\geq 0}}

\def\CUBE{C}

\def\LOPER{\mathcal{L}}
\def\LOPA{\LOPER_1}
\def\LOPB{\LOPER_2}
\def\LOPAm{\LOPER_{1,m}}
\def\LOPBm{\LOPER_{2,m}}
\def\DOPER{\mathcal{D}}

\def\Oo{\mathcal{O}}

\def\N{\mathbb{N}}
\def\R{\mathbb{R}}

\def\Z{\mathbb{Z}}

\def\DT{h}

\def\EXP#1{e^{#1}}

\def\EXPECT{{\mathbb{E}}}

\def\PROB#1{{\mathbb{P}\left({#1}\right)}}

\def\SPACE{\;\;\;}          
\def\COMMA{\,,}             
\def\PERIOD{\,.}            
\def\SEP{{\,|\,}}           

\def\VIZ#1{(\ref{#1})}      

\def\BIGO{\Oo}
\def\SMALLO{\mathrm{o}}

\newtheorem{remark}{Remark}[section]
\newtheorem{example}{Example}[section]

\title{%
  Parallelization, processor communication and error analysis in lattice kinetic Monte Carlo\thanks{%
   The research of G.A. was partially supported by the National Science Foundation under the grant
   NSF-DMS-071512 and NSF-CMMI-0835582.
   The research of M.A.K. was partially supported by  the Office of Advanced Scientific Computing Research,
   U.S. Department of Energy under DE-SC0002339 and the European Commission FP7-REGPOT-2009-1 Award No 245749.
   The research of P.P. was partially supported by the National Science Foundation under the grant
   NSF-DMS-0813893 and by the Office of Advanced Scientific Computing Research,
   U.S. Department of Energy under DE-SC0001340.
                              }
}
\author{%
Giorgos Arampatzis\thanks{%
Department of Applied Mathematics,
University of Crete and  Foundation of Research and Technology-Hellas, Greece,
{\tt garab@math.uoc.gr}}
\and
Markos A. Katsoulakis\thanks{%
Department of Mathematics and Statistics, University of Massachusetts,
Amherst, MA 01003--9305, USA,
{\tt markos@math.umass.edu}}
\and
Petr Plech\'a\v{c}\thanks{%
Department of Mathematical Sciences, University of Delaware, Newark, DE 19716, USA,
{\tt plechac@math.udel.edu}}
}
\begin{document}

\maketitle

\begin{abstract}
In this paper we study from a numerical analysis perspective the Fractional Step Kinetic Monte Carlo (FS-KMC) 
algorithms proposed in \cite{AKPTX} for the parallel simulation
of spatially distributed particle systems on a lattice. 
FS-KMC  are fractional step algorithms  with a time-stepping  window $\Delta t$, and as such they
are inherently {\em partially asynchronous} since there is no processor communication during the
period $\Delta t$.
In this contribution we  primarily focus on the error analysis of FS-KMC algorithms as approximations of
conventional, serial kinetic Monte Carlo (KMC).
A key aspect of our analysis relies on emphasizing a goal-oriented approach for suitably defined macroscopic observables (e.g.,
density, energy, correlations, surface roughness), rather than 
focusing on strong topology estimates for individual trajectories.

One of the key  implications of  our error analysis is that it allows us to address systematically the processor communication of
different parallelization strategies for KMC by  comparing their (partial) asynchrony, which in turn is measured by their respective  
fractional time step $\Delta t$ for a prescribed  error tolerance.
\end{abstract}



\section{Introduction}

The simulation of stochastic lattice systems using kinetic Monte Carlo (KMC) methods relies on the direct numerical simulation of the underlying 
Continuous Time Markov Chain (CTMC).
In \cite{AKPTX} we proposed  a new mathematical and computational framework for constructing parallel algorithms for KMC simulations.%
The  parallel algorithms in \cite{AKPTX} are controlled approximations of Kinetic Monte Carlo
algorithms, and  rely on first  developing 
a spatio-temporal decomposition of  the Markov operator for the underlying CTMC   into  a hierarchy of operators 
corresponding  to the particular   parallel architecture.  Based on this operator decomposition, we  formulated  
{\it Fractional Step Approximation schemes}
by employing the Trotter product formula, which in turn determines the processor {\em communication schedule}. 
The fractional step framework  allows for a hierarchical structure to be easily formulated and implemented, offering  
a key  advantage for simulating on modern parallel architectures with elaborate memory and processor hierarchies. 
The resulting parallel algorithms are inherently {\em partially asynchronous}
as processors do not communicate during the fractional time step window $\Delta t$.

Earlier,  in \cite{ShimAmar05b} the authors also proposed an {\em approximate} algorithm,   
in order to create a parallelization scheme for KMC. It was  demonstrated in \cite{ShimAmar09, SPPARKS}, that %
boundary inconsistencies  are resolved in a straightforward fashion, while there is an absence of global communications. Finally, among the parallel algorithms  
tested in \cite{ShimAmar09}, the  one in 
 \cite{ShimAmar05b} had  
the highest parallel efficiency. In \cite{AKPTX}, we demonstrated that the approximate algorithm in
\cite{ShimAmar05b} is a special case of the  Fractional Step Approximation 
schemes introduced in \cite{AKPTX}.  There we also demonstrated,  using the Random Trotter Theorem  \cite{Kurtz},  that the algorithm in  \cite{ShimAmar05b} is {\em numerically consistent} 
in the approximation limit, i.e.,  as the time step in the fractional step scheme converges to zero, it   converges to a Markov Chain  that has the same 
master equation and generator as the original serial KMC.
The open source SPPARKS parallel Kinetic Monte Carlo simulator, \cite{SPPARKS}, can also be formulated as a Fractional Step approximation.
In this article, the convergence, reliability and efficiency of  all such Fractional Step KMC parallelization methods is systematically explored  by rigorous  numerical analysis 
which relies on controlled-error approximations  in transient regimes relevant to the  simulation of extended systems. 

A key aspect of 
the presented analysis relies on emphasizing a {\em goal-oriented} error approach  for suitably defined macroscopic observables, e.g.,
density, energy, correlations, surface roughness , 
giving rise to estimates which are independent of the (very large) system size of the particle system.
Besides the obtained numerical consistency and reliability of the approximating CTMC obtained from FS-KMC
there is  an additional key practical point: the bigger is the allowable $\Delta t$, within  a desired  error tolerance,   the less processor communication is required.
From a broader mathematical perspective, and driven from parallel computing challenges,  the  developed mathematical and  numerical analysis
attempts to balance between controlled error approximations and  processor communication.
The same methods could also prove useful for developing and evaluating  parallel numerical schemes  for other molecular  and extended systems.

\section{Background}

Kinetic Monte Carlo (KMC) algorithms have proved to be an important tool for the simulation of 
{\em non-equilibrium, spatially distributed} chemical processes arising in applications 
ranging from materials science, catalysis and reaction engineering, to complex biological processes. 
Typically the simulated models involve chemistry and/or transport 
micro-mechanisms for atoms and molecules, e.g.,
reactions, adsorption, desorption processes and diffusion on surfaces and through  porous media,  
\cite{binder, Auerbach, ACDV}. Furthermore, mathematically similar mechanisms and 
corresponding KMC simulations arise in {\em agent-based} models in epidemiology, ecology and traffic networks, \cite{Szabo}.

We consider an interacting particle system defined on a $d$-dimensional lattice $\LATT$.
Naturally, the simulations are performed on a finite lattice of the size $N$, however, given the size of real
molecular systems it is either necessary to treat the case $N\to\infty$, e.g., $\Lambda = \Z^d$,
or alternatively any numerical estimates we obtain need to be {\em independent of the system size} $N$.
We restrict our discussion to lattice gas models where the order parameter or the spin variable
takes values in a compact set, in most cases the set is finite  $\SPINSP=\{0,1,\dots, \NUMSP \}$.
At each lattice site $x\in \LATT$ an order parameter (a spin variable) $\sigma(x)\in \SPINSP$ 
is defined. The states in $\SPINSP$ correspond to occupation of the site $x\in\LATT$ by different
species. For example, if $\SPINSP=\{0,1\}$ the order parameter models the classical lattice gas with
a single species occupying the site $x$ when $\sigma(x)=1$ and with the site being vacant if $\sigma(x)=0$.
We denote $\PROCMIC$ the stochastic process with values in the countable configuration space  $\SIGMA=\SPINSP^{\LATT}$.
Microscopic dynamics is described by transitions (changes) of spin variables at different sites. 
We study systems in which the transitions are localized and involve only finite number of sites at each
transition step. First, the {\em local} dynamics is described by an updating mechanism and corresponding  transition rates
$c(x,\omega;\sigma)$ in \VIZ{rates0}, such that 
the configuration at time $t$,  $\sigma_t=\sigma$ changes into a new configuration $\CONFNEW$ by an
update in a neighborhood of  the site $x\in\LATT$. Here $\omega\in\SIGMA_x$, where $\SIGMA_x$ is the set of all possible
configurations that correspond to an update at a neighborhood $\Omega_x$ of the site $x$. For example, 
if the modeled process is 
a diffusion of the classical lattice gas a particle at $x$, i.e., $\sigma(x)$ can move to any unoccupied 
nearest neighbor $y$ of $x$, i.e., $\Omega_x = \{y\in\LATT\SEP |x-y|=1\}$ and $\SIGMA_x$ is the set of all possible
configurations $\SIGMA_x = \SPINSP^{\Omega_x}$. Computationally the sample paths $\PROCMIC$
are constructed 
via KMC, that is  through the procedure described in \VIZ{totalrate} and \VIZ{skeleton} below.

The studied stochastic processes are set on a lattice (square, hexagonal, etc.) $\LATT$ with $N$ sites, they have a 
discrete, albeit high-dimensional, configuration space  $\SIGMA$ and necessarily have to be of jump type describing transitions between different 
configurations $\sigma \in \SIGMA$. 
Mathematically, a CTMC  is a stochastic process $\{\sigma_t\}$ defined completely in terms of  the
local  transition rates $c(\sigma, \sigma')$ which determine  the  updates (jumps)  from any current
state $\sigma_t=\sigma$ to a (random) new state $\sigma'$.
In the context of the spatially distributed applications in which we are interested  here, the  local transition rates will be denoted as
\begin{equation}\label{rates0}
c(\sigma, \sigma')=c(x,\omega;\sigma)\COMMA
\end{equation}
which correspond to an updating micro-mechanism 
from  a current configuration $\sigma_t=\sigma$ of the system to a new configuration $\CONFNEW$ by performing an update
in a neighborhood of  each site $x\in\LATT$. Here $\omega$ %
is an index  for all possible
configurations $\SIGMA_x$ that correspond to an update at a neighborhood $\Omega_x$ of the site $x$; we refer 
to the end of the section for specific examples.

The probability of a transition over an infinitesimal time interval $\DELTAT$ is
$$
\PROB{\sigma_{t+\DELTAT} = \CONFNEW\SEP \sigma_t = \sigma} = c(x,\omega;\sigma)\DELTAT +
\SMALLO(\DELTAT) \PERIOD
$$
Realizations of the process are constructed from the embedded discrete time  Markov chain $S_n =
\sigma_{t_n}$ (see \cite{KL}), corresponding to jump times $t_n$. 
The local transition rates \VIZ{rates0} define the total rate
\begin{equation}\label{totalrate}
  \lambda(\sigma)=\sum_{x \in \LATT}\sum_{\omega \in \SIGMA_x } c(x, \omega; \sigma)\COMMA
\end{equation}
which is the intensity of the exponential waiting time for a jump from the state $\sigma$. 
The transition probabilities for the embedded Markov chain $\{S_n\}_{n\geq 0}$ are 
\begin{equation}\label{skeleton}
   p(\sigma, \CONFNEW)=\frac{c(x, \omega; \sigma)}{\lambda(\sigma)}\PERIOD
\end{equation}
In other words once the exponential ``clock'' signals a jump, the system transitions from the state $\sigma$ 
to a new configuration $\CONFNEW$ with the probability $p(\sigma, \CONFNEW)$.
On the other hand, the evolution of the entire system at any time $t$ is described by the  transition probabilities
$P(\sigma, t ; \zeta):=\PROB{\sigma_{t} = \sigma\SEP \sigma_0 = \zeta}$
where $\zeta \in \SIGMA$ is an  initial configuration. The transition probabilities corresponding to the local rates 
\VIZ{rates0} satisfy the Forward Kolmogorov Equation ({Master Equation}), \cite{Liggett,Gardiner04},
\begin{equation}  \label{master}
  \partial_t P(\sigma, t ; \zeta):=\sum_{\sigma',
  \sigma'\ne \sigma} c(\sigma', \sigma)P(\sigma', t; \zeta)-\lambda(\sigma)P(\sigma, t ; \zeta)\COMMA
\end{equation}
where $P(\sigma, 0 ; \zeta)=\delta (\sigma-\zeta)$ and $\delta (\sigma-\zeta)=1$ if $\sigma=\zeta$
and zero otherwise.
In \cite{AKPTX}  we  developed a general mathematical framework for {\em parallelizable  approximations} of the KMC algorithm.
Rather than focusing on exactly constructing stochastic trajectories in \VIZ{totalrate} and \VIZ{skeleton}, 
we proposed to approximate the evolution of {\em observables} $f=f(\sigma) \in \CBS$, i.e., of bounded continous functions on
the configuration space $\SIGMA$.
The space of bounded continuous functions, $\CBS$, is regarded
as a Banach space with the norm
$$
  \NORMB{f} = \sup_{\sigma\in\SIGMA} |f(\sigma)|\PERIOD
$$
Here we consider observables/functions $f(\sigma)$ depending on large  number of variables $\sigma(x)$, $x\in\LATT$, 
such as 
coverage, surface roughness, correlations, etc., see for instance  the examples in Section~\ref{macroscopic}.
Alternatively, we may consider observables depending on infinitely many variables $\sigma(x)$, $x\in\Lambda=\Z^d$, 
to stress the necessity of working with the infinite volume limit.

Typically in KMC we need to compute expected values of such observables, that is quantities such as
\begin{equation}\label{propagator1}
   u(\zeta, t):=\EXPECT^\zeta[f(\sigma_t)]=\sum_\sigma f(\sigma)P(\sigma, t; \zeta)\COMMA
\end{equation}
conditioned on the initial data $\sigma_0 = \zeta$.
By a straightforward calculation using \VIZ{master} we obtain that  the observable \VIZ{propagator1}
satisfies the initial value problem
\begin{equation}\label{ODE}
    \partial_t u(\zeta, t)=\LOPER u(\zeta, t)\, , \quad\quad u(\zeta, 0)=f(\zeta)\COMMA
\end{equation}
where the operator $\LOPER:\CB(\SIGMA) \to \CB(\SIGMA)$ is known as the {\em generator} of the continous time Markov chain, \cite{Liggett},
and in the case of \VIZ{rates0} it is
\begin{equation}\label{generator}
    \LOPER f(\sigma) = \sum_{\sigma'}c(\sigma, \sigma')[f(\sigma')-f(\sigma)]=\sum_{x\in\LATT}\sum_{\omega\in\SIGMA_x} c(x,\omega;\sigma) [f(\CONFNEW) - f(\sigma)]\PERIOD
\end{equation}
We then write \VIZ{propagator1}, as the the action of the Markov semi-group  $\EXP{t\LOPER}$
associated with the generator $\LOPER$ and the process $\PROC{\sigma_t}$, \cite{Liggett},
on the observable $f$
\begin{equation}\label{semigroup}
    u(\zeta, t)=\EXPECT^\zeta[f(\sigma_t)]=\EXP{t\LOPER}f(\zeta)\COMMA
\end{equation}
where $\EXPECT^{\zeta}$ denotes the expected value with respect to the law %
of the process $\{\sigma_t\}_{t\geq 0}$ conditioned
on the initial configuration $\zeta$.

We define a difference operator $\DX f$ as an analogue of a derivative.  
Higher-order derivative analogues are defined in Section~\ref{macroscopic} when needed in the error analysis.
We define  a corresponding function space,
which is necessary in order to set up the semigroup $\SEMIGRP=\EXP{t\LOPER}$ when we consider the infinite lattice 
$\Lambda=\Z^d$ or to obtain estimates which are independent of the system size $N$ when considering the lattice 
$\LATT$ in Section~\ref{macroscopic}.
\begin{definition}\label{definition:derivative}
 Let $f\in\CBS$ then for any $x\in\LATT$ we define
 $$
    \DXOMEGA f(\sigma) = f(\CONFNEW) - f(\sigma)\PERIOD
 $$
We define the norm 
$$ \NORMA{1}{f} \equiv \sum_{x,\omega} \NORM{\DXOMEGA f} $$
and the space of functions on $\SIGMA=\SPINSP^{\LATT}$
$$
  \CSS{1} = \{ f\in\CB(\SIGMA)\SEP \NORMA{1}{f} \leq C_f \;\;\mathrm{where}\;\; C_f 
               \;\;\mathrm{is\; independent \; of}\;\; N \ \}\PERIOD
$$
Similarly we define the space of functions on $\SIGMA=\SPINSP^{\Lambda}$ associated with the infinite lattice $\Lambda= \Z^d$
$$
  \CSS{1} = \{ f\in\CB(\SIGMA)\SEP \NORMA{1}{f} < \infty \}\PERIOD
$$
\end{definition}
Because of the estimates in Section~\ref{macroscopic}, see \VIZ{phi_norm} and \VIZ{theta_norm} in
Theorem~\ref{bound:L2_on_u}, we will employ spaces with higher  discrete derivatives that will be defined
in Section~\ref{macroscopic}.
On the infinite lattice $\Lambda$  macroscopic observables are all $f \in \CSS{1}$. In  the case of
$\LATT$, macroscopic observables are 
all $f=f(\sigma)$ such that $\NORMA{1}{f}$ is independent of the system size $N$; such typical
examples are discussed in Section~\ref{macroscopic}.

Typically, the evolution of the particle system on the infinite lattice $\Lambda=\Z^d$ 
is  well-defined, as demonstrated in the
next propositions.
\begin{proposition} 
For any $f \in \CSS{1}$ we have that the series 
$$
  \LOPER f(\sigma) = \sum_{x\in\Lambda}\sum_{\omega\in\SIGMA_x} \RATEC [f(\CONFNEWX) - f(\sigma)]\COMMA
$$
converges uniformly and defines a function in $\CB(\SIGMA)$, provided $ \sup_{\omegax, \sigma}  c(x, \omega; \sigma)<\infty$. 
Furthermore,
$$
\NORMB{\LOPER f} \leq \sup_{\omegax, \sigma}  c( x, \omega; \sigma)
\NORMA{1}{f}\PERIOD
$$
\end{proposition}
\begin{proof} Follows directly from \VIZ{generator} and the definition of $\CSS{1}$.
\end{proof}

\begin{proposition}\label{SemigroupDef}
Under the boundedness assumptions on the rates, the closure of the operator $\LOPER$ defines a Markov generator for a Markov
semigroup $\SEMIGRP\equiv \EXP{t\LOPER}$, such that for $f\in\CSS{1}$,  $\SEMIGRP f \in\CSS{1}$ and
$$
\NORMA{1}{\SEMIGRPt f} \leq \EXP{\Gamma  t} \NORMA{1}{f}\COMMA
$$
where $\Gamma$ is a constant depending on the rates $c(x, \omega; \sigma)$.
\end{proposition}

\begin{proof}
See \cite[Theorem 3.9, pp 27]{Liggett}.
\end{proof}

Clearly the same results hold for the finite lattice $\LATT$ and the corresponding high-dimensional configuration space $\SIGMA$, 
where all constants are independent of the size $N$.

\medskip

\noindent{\sc Examples.}\\
\noindent{\it Adsorption/Desorption for single species particles.} In this case spins take values in
       $\sigma(x)\in\SPINSP=\{0,1\}$, $\Omega_x = \{x\}$, $\SIGMA_x = \{0,1\}$ and the update represents a spin flip
       at the site $x$, i.e., for $z\in\LATT$
       $$
	  \CONFNEW(z) \equiv \sigma^x(z) =  \begin{cases} 
	                                       \sigma(z) & \mbox{if $z\neq x$,} \\
                                               1 - \sigma(x) & \mbox{if $z = x$.}
	                                    \end{cases}
       $$
\smallskip

\noindent{\it Diffusion for single species particles.} The state space for spins is $\sigma(x)\in\SPINSP=\{0,1\}$, 
       $\Omega_x = \{y\in\LATT\SEP |x-y|=1\}$ includes all nearest neighbors of the site $x$ to which a particle
       can move. Thus the new configuration $\CONFNEW = \sigma^{(x,y)}$ is obtained by updating the
       configuration $\sigma_t=\sigma$ 
       from the set of possible local configuration changes $\{0,1\}^{\Omega_x}$ using the specific rule, also known
       as spin exchange, which involves changes at two sites $x$ and $y\in\Omega_x$
       $$
	  \CONFNEW(z) \equiv \sigma^{(x,y)}(z) =  \begin{cases} 
	                                              \sigma(z) & \mbox{if $z\neq x,y$,} \\
                                                      \sigma(x) & \mbox{if $z = y$,} \\
                                                      \sigma(y) & \mbox{if $z = x$.}
	                                         \end{cases}
       $$
       The transition rate is then written as $c(x,\omega;\sigma) = c(x,y;\sigma)$.
       The resulting process $\PROCMIC$ defines  dynamics with the total number of particles ($\sum_{x\in\LATT}\sigma(x)$)
       conserved, sometimes referred to as Kawasaki dynamics, \cite{ACDV}.
\smallskip

\noindent{\it Multicomponent reactions.} Reactions that involve $\NUMSP$ species of particles are easily described 
      by enlarging the spin space to $\SPINSP=\{0,1,\dots,\NUMSP\}$. If the reactions occur only at a single site
      $x$, the local configuration space $\SIGMA_x = \SPINSP$ and the update is indexed by $k\in\SPINSP$ with the
      rule
      $$
	  \CONFNEW(z) \equiv \sigma^{(x,k)}(z) =  \begin{cases} 
	                                              \sigma(z) & \mbox{if $z\neq x,y$,} \\
                                                      k  & \mbox{if $z = x$.}
	                                         \end{cases}
      $$
      The rates $c(x,\omega;\sigma) \equiv c(x,k;\sigma)$ define probability of a transition $\sigma(x)$ to species
      $k=1,\dots,\NUMSP$ or vacating a site, i.e., $k=0$, over $\DELTAT$.

      \medskip

\smallskip

\noindent{\it Reactions involving  particles with internal degrees of freedom.}
      Typically a reaction involves particles with internal degrees of freedom, and in this case  several neighboring lattice sites 
      may be updated at the same time, corresponding to the degrees of freedom of the  particles involved in the reaction. 
      For example, in a case such as  CO oxidation on a catalytic surface, \cite{evans09}, 
      when only particles at a nearest-neighbor distance can react we set $\sigma(x)\in\SPINSP=\{0,1,\dots,\NUMSP\}$,
      $\Omega_x = \{y\in\LATT\SEP |x-y|=1\}$ and the set of local updates $\SIGMA_x = \SPINSP^{\Omega_x}$. Such $\SIGMA_x$
      contains all possible reactions in a neighborhood of $x$. When reactions involve only pairs of species, the rates can be
      indexed by $k$, $l\in \SPINSP$, or equivalently $\SIGMA_x = \SPINSP\times \SPINSP$. 
      Then the reaction rate $c(x,\omega;\sigma) = c(x,y,k,l;\sigma)$ describes the  probability per unit
      time of  $\sigma(x)\to k$ at the site $x$ and $\sigma(y)\to l$ at $y$, i.e., the updating mechanism
      $$
	  \CONFNEW(z) \equiv \sigma^{(x,y,k,l)}(z) =  \begin{cases} 
	                                              \sigma(z) & \mbox{if $z\neq x,y$,} \\
                                                      k         & \mbox{if $z = x$,} \\
                                                      l         & \mbox{if $z = y$,}
	                                         \end{cases}
      $$
      where $|x-y|=1$.
\section{Towards parallel kinetic Monte Carlo algorithms}
In practice, the sample paths $\PROCMIC$
are constructed by the kinetic Monte Carlo algorithm, that is  by simulating the embedded Markov chain defined by 
\VIZ{totalrate} and \VIZ{skeleton}
and advancing the tine by random time-steps from the exponential distribution. 
Implementations are based on the efficient calculation of transition probabilities, e.g.,  \cite{BKL75} for 
Ising models, known as a BKL Algorithm,  and in \cite{Gillespie76} known as Stochastic Simulation Algorithm (SSA) for reaction systems.

It is evident from formulas \VIZ{totalrate} and \VIZ{skeleton}, that KMC algorithms are {\em inherently serial} as updates are done 
at one site $x \in \LATT$ at a time, while on the  other hand  \VIZ{totalrate}  depends on information from the entire spatial 
domain $\LATT$. For these reasons it appears that KMC  cannot be parallelized easily.
However, Lubachevsky, in \cite{Lubachevsky88}, proposed  an asynchronous  approach for parallel KMC simulation  in the context of Ising systems, 
in the sense that different processors simulate independently parts of the physical domain, while  inconsistencies 
at the boundaries are corrected with a series of suitable rollbacks. This method relies on the uniformization of \VIZ{totalrate}; 
thus the approach yields a  null-event algorithm, \cite{binder},  which includes 
rejected moves over the entire spatial sub-domain that corresponds to  each processor, see also \cite{Nicol}.
A modification in order to incorporate the BKL Algorithm  was proposed in \cite{Lubachevsky88}, which was implemented 
and tested in \cite{Korniss99}.  This is a still asynchronous algorithm, where  BKL-type  rejection-free simulations are carried out 
in the interior of each sub-domain (processor), while uniform rates are used at the boundaries, reducing rejections to just the boundary set. 
However, these asynchronous algorithms 
may still have a high number of rejections for boundary events and rollbacks, which considerably  
reduce the parallel efficiency, \cite{ShimAmar05}. 
Advancing processors in time in a synchronous manner over a fixed time-window can provide a way to mitigate the excessive number of 
boundary inconsistencies 
between processors and ensuing rejections and rollbacks in earlier methods. Such {\em synchronous} parallel KMC algorithms were proposed  
in \cite{Lubachevsky93}, \cite{ShimAmar05}, 
\cite{MerickFichthorn07}, \cite{ShimAmar09}. However, several costly {\em global} communications are required at each cycle 
between all processors whenever a boundary event occurs in any one of them, in order to avoid errors in processor communication, \cite{ShimAmar09}.
As we will discuss  in the sequel, many of these issues with parallel KMC can be addressed by abandoning the 
earlier perspective on creating a parallel KMC algorithm with exactly the same 
rates $c(x, \omega; \sigma)$ in  \VIZ{generator} as the serial algorithm.

Indeed, in  \cite{AKPTX},  we  adopted  the approach of creating a parallel KMC algorithm which {\em approximates}  the underlying  
continuous time Markov chain
of the serial algorithm instead of reproducing its master equation exactly. 
We proposed  a spatio-temporal decomposition for  the Markov operator underlying  the KMC  algorithm  into  a hierarchy of operators  
corresponding  to the processor architecture. Based on this operator 
decomposition we  can formulate {\it Fractional Step  KMC Approximation} schemes
by employing the Trotter product formula. In turn these approximating schemes  determine the {\em Communication Schedule} between processors through 
the sequential application of the operators in the decomposition, as well as  the time step employed in the particular fractional step scheme.
Earlier,  in \cite{ShimAmar05b} the authors also proposed an {\em approximate} algorithm,   
in order to create a parallelization scheme for KMC. It was  demonstrated in \cite{ShimAmar09, SPPARKS}, that %
boundary inconsistencies  are resolved in a straightforward fashion, while there is an absence of global communications. Finally, among the parallel algorithms  
tested in \cite{ShimAmar09}, the  one in 
 \cite{ShimAmar05b} had  
the highest parallel efficiency. In \cite{AKPTX}, we demonstrated that the approximate algorithm in
\cite{ShimAmar05b} is a special case of the {\it Fractional Step Approximation} 
schemes introduced
in \cite{AKPTX}. We also demonstrated,  using the Random Trotter Theorem,  \cite{Kurtz},  that the algorithm in  \cite{ShimAmar05b} 
is {\em numerically consistent} in the approximation limit, i.e.,  as the time step in the fractional step scheme converges to zero, 
it   converges to a Markov Chain  that has the same master equation and generator as the original serial KMC.
Finally, the open source SPPARKS parallel Kinetic Monte Carlo simulator, \cite{SPPARKS}, also relies on such Fractional Step approximations.

In this article, the convergence, reliability and efficiency of parallel algorithms, that fit the Fractional Step KMC approximation 
framework, are systematically explored  by rigorous  numerical analysis which relies on controlled-error approximations 
in transient regimes relevant to the  simulation of extended systems.

\subsection{Fractional time step kinetic Monte Carlo algorithms}\label{schemes}

In \cite{AKPTX} we proposed a class of parallel KMC algorithms that are based on operator splitting 
of the Markov generator $\LOPER$ which is based
on a geometric decomposition of the lattice $\LATT$. 
\begin{definition}\label{decomposition}
The lattice $\LATT$ is decomposed into non-overlapping coarse cells $\CUBE_m$, $m=1,\dots,M$ such that, $|C_m|=Q=q^d$, where $d$ is the dimension, 
\begin{equation}\label{decomp}
\LATT = \bigcup_{m=1}^M \CUBE_m\COMMA\;\;\;\; \CUBE_m \cap \CUBE_n = \emptyset\COMMA\; m\neq n\COMMA\; N=MQ \PERIOD
\end{equation}
The range of interactions is defined as $L = \max_{x\in\CUBE_m}\{\DIAM \Omega_x\}$. 
For a coarse cell $\CUBE_m$ the closure of this set is
$$
\CLSR\CUBE_m = \{ z\in\LATT\SEP |z - x|\leq L \COMMA x\in\CUBE_m\}\PERIOD
$$
The boundary of $\CUBE_m$ is then defined as $\BNDRY\CUBE_m = \CLSR\CUBE_m \;\setminus\;\CUBE_m$. 
\end{definition}

The closure $\CLSR\CUBE_m$ thus includes all sites of $\CUBE_m$ and  all ``boundary'' 
lattice sites $\BNDRY\CUBE_m$  which are connected with sites in $\CUBE_m$ through  particle interactions in the  updating mechanism, 
see Figure~\ref{fig:lattice_partition}. 
In many models the value of the interaction range $L$ is independent of $x$ due to the translational invariance of the model. 
This geometric partitioning induces a decomposition of \VIZ{generator}
\begin{equation}\label{gendecomp}
\LOPER f(\sigma) = \sum_{m=1}^M \LOPER_m f(\sigma)\,, \quad \quad\LOPER_m f(\sigma)=
                  \sum_{x\in\CUBE_m} \sum_{x,\omega\in\SIGMA_x} c(\omega;\sigma) [f(\CONFNEW) - f(\sigma)]\PERIOD          
\end{equation}
The generators $\LOPER_m$ 
define a new Markov process $\{\sigma^m_t\}_{t\geq 0}$ on the {\it entire} lattice $\LATT$.
\begin{remark}
{\rm
In many models  such as  in catalysis the interactions between particles are  short-range, \cite{Reuter1, evans09},
and therefore the transition rates $c(x,\omega;\sigma)$ depend on the configuration $\sigma$
only through
$\sigma(x)$ and $\sigma(y)$ with $|x-y| \le L$, where $L$ is small (typically one). Similarly the new configuration $\CONFNEW$ 
involves changes only at the sites in this neighborhood. 
Thus the generator $\LOPER_m$ updates the lattice sites at most in the set 
$\CLSR \CUBE_m = \{ z\SEP |x-z|\leq L\COMMA x\in\CUBE_m\}$. 
Consequently the processes $\PROC{\sigma^m_t}$ and $\PROC{\sigma^{m'}_t}$
corresponding to $\LOPER_m$ and $\LOPER_{m'}$ are {\em independent provided} $\CLSR\CUBE_m \cap \CLSR\CUBE_{m'} = \emptyset$.
}
\end{remark}
\begin{figure}[ht]
  \centerline{%
        \includegraphics[scale=0.5]{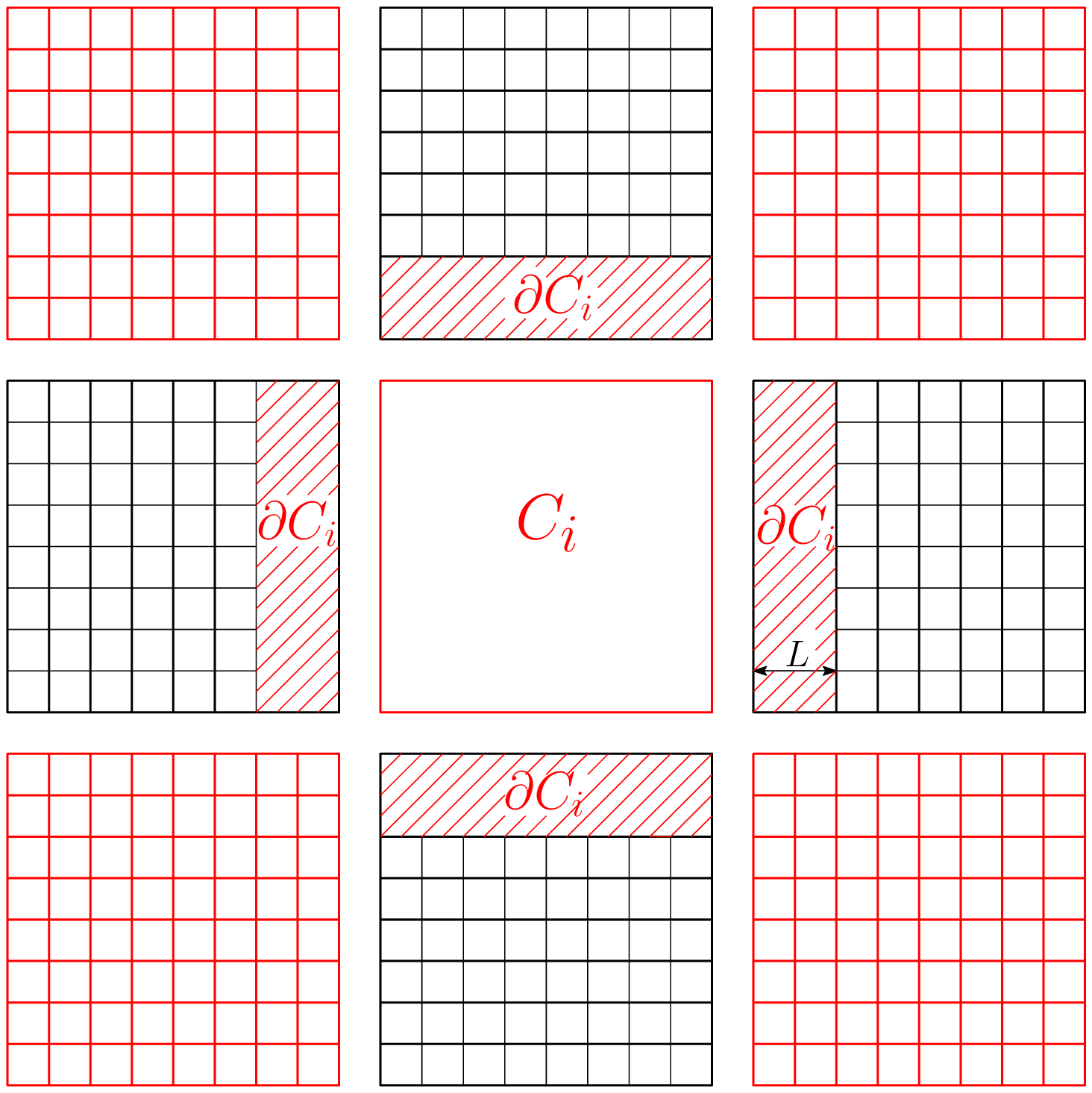}
  }%
  \caption{\label{fig:lattice_partition} Lattice partitioning in \VIZ{opdecomp}. Note that later we
           use the notation $C_i^\partial$ to denote the interior boundary of $C_i$, see
           Figure~\ref{fig:sublattice}}
\end{figure}
The operator decomposition   yields an algorithm suitable
for parallel implementation, in particular, in the case of short-range interactions when the communication
overhead can be handled efficiently: if the lattice $\LATT$ is partitioned into subsets $\CUBE_m$ such that
$\DIAM\CUBE_m > L$,  we can group the sets
$\{\CUBE_m\}_{m=1}^M$ so that there is no interaction between sites in  $\CUBE_m$ that
belong to the same group.
For the sake of simplicity we assume that the lattice is divided into two {\em sub-lattices}   described
by the index sets $\INDXE$ and $\INDXO$ (black/red in each block in Fig.~\ref{fig:lattice_partition}), 
which in turn {\em induce}  a corresponding splitting of the generator:
\begin{eqnarray}
  && \LATT = \LATTA \cup\LATTB:=\bigcup_{m\in\INDXE} \CUBE_m \cup \bigcup_{m\in\INDXO}
     \CUBE_m \;\;\; \textrm{and}\nonumber\\
  && \LOPER =\LOPA+\LOPB:= \sum_{m\in\INDXE} \LOPAm + \sum_{m\in\INDXO} \LOPBm
     \label{opdecomp}\PERIOD
\end{eqnarray}

The decomposition \VIZ{opdecomp}  has key consequences for simulating the process $\PROC{\sigma_t}$ in parallel, 
as well as formulating different  related algorithms.
The processes $\PROC{\sigma^m_t}$ corresponding to the  generators $\LOPAm$
are mutually independent  for different $m\in\INDXE$, and thus can be simulated in parallel.
Similarly  we can handle the processes belonging to the
group indexed by $\INDXO$. However, there is still {\em local communication/synchronization} between these two groups
as there is non-empty overlap between the groups due to interactions and updates in the sets
$\CLSR\CUBE_m\cap\CLSR\CUBE_{m'}$ when $m\in\INDXE$ and $m'\in\INDXO$ and the cells are within the interaction range $L$. 
Mathematically, we can describe all that through a  fractional step approximation of 
the Markov semigroup $\SEMIGRP\equiv\EXP{t\LOPER}$ of the process $\PROC{\sigma_t}$. The operator
splitting or equivalently the fractional step
approximation can be also viewed as an alternating dimension approximation since we solve the evolution of
$u(\sigma,t)$ given as solution of \VIZ{ODE} by alternating between evolution of $\sigma$'s in the dimensions corresponding 
to $\INDXE$ and $\INDXO$. 

Indeed, the key tool for our analysis are 
different   versions of the Trotter formula, \cite{Trotter, Kurtz},
\begin{equation}\label{trotter}
 \EXP{T\LOPER} = \lim_{n\to\infty}\left[ \EXP{\frac{T}{n}\LOPA}\EXP{\frac{T}{n}\LOPB}\right]^n
 \end{equation}
when applied to the operator $\LOPER = \LOPA + \LOPB$ in \VIZ{opdecomp}.
Thus to reach a time $T$ we define
a time step $\Delta t = \DT = \tfrac{T}{n}$ for a fixed value of $n$ and alternate the evolution by
$\LOPA$ and $\LOPB$,  giving rise to 
the {\em Lie } splitting approximation for $n\gg 1$:
\begin{equation}\label{lie}
   \EXP{T\LOPER} \approx  \SEMIGRPL:=\left[ \EXP{\Delta t\LOPA}\EXP{\Delta t\LOPB}\right]^n, \SPACE
   \mathrm{where} \SPACE \Delta t = \tfrac{T}{n} \PERIOD
\end{equation} 
To develop a   parallelizable  scheme we use  the fact that the action of the operator $\LOPA$ (and
similarly of $\LOPB$) can be distributed onto independent processing units, indexed by $m$ in \VIZ{opdecomp}, 
$$
     \EXP{\Delta t\LOPA}=\prod_{{m\in\INDXE}}\EXP{\Delta t\LOPAm}\COMMA 
    \quad\quad \EXP{\Delta t\LOPB}=\prod_{{m\in\INDXO}}\EXP{\Delta t\LOPBm}\PERIOD
$$
Analogously we have the {\em Strang} splitting scheme
\begin{equation}\label{strang}
  \EXP{T\LOPER} \approx \SEMIGRPS:=\left[\EXP{\frac{\Delta t}{2}\LOPA}\EXP{{\Delta t}\LOPB}\EXP{\frac{\Delta t}{2}\LOPA}\right]^n\COMMA
   \SPACE \mathrm{where} \SPACE \Delta t = \tfrac{T}{n} \PERIOD
\end{equation}      
From now on, for the notational convenience, we shall also use $\DT$ to symbolize $\Delta t$.

While operator splitting has been exploited in many classical numerical methods, e.g.,\cite{Hairer}, in our context it offers a rigorous 
framework for extending simple (deterministic) alternating strategies associated with, for example, traditional {\it Lie or Strang} splittings 
to more elaborate and randomized {\it Processor Communication Schedules}, we refer to Section~\ref{PCS} for a complete discussion.

We characterize the FS-KMC (Fractional Step KMC) algorithm \VIZ{lie} as {\em partially asynchronous}
since there is no processor communication during 
the period $\DT$. Furthermore, at every $\DT$
we have only local synchronization between processors, i.e., between the sets
$\CLSR\CUBE_m\cap\CLSR\CUBE_{m'}$ when $m\in\INDXE$ and $m'\in\INDXO$.  Hence, 
the bigger the allowable $\DT$ in \VIZ{lie} or in \VIZ{strang} the less processor communication we have, in which case 
the error in the approximation \VIZ{lie} or \VIZ{strang} worsens.
This {\em balance between accuracy  and processor communication} in algorithms is one of the themes of this article.

\section{Local and global error analysis}
The FS-KMC algorithm approximates the evolution of observables $u(\sigma,t)$ given by the original semigroup $\SEMIGRP$. 
We present an error analysis which focuses on  classes of observables such as \VIZ{propagator1} instead on estimating an approximation
of the  probability distribution of the process solving \VIZ{master}. This perspective is also relevant to  practical simulations, 
where the estimated quantity is linked to specific observables, and is 
simulated  by the FS-KMC algorithm.

To understand the error of this approximation we first analyze
the error for the two cases of deterministic PCS: the Lie splitting defines a new semigroup \VIZ{lie} that we denote 
$\SEMIGRPL$ and similarly $\SEMIGRPS$ denotes
the semigroup \VIZ{strang} obtained by the Strang splitting. The local error analysis can be treated in a similar way 
as it is done for the  finite
dimensional case when working on the lattice $\LATT$ by using the property proved in \cite{jahnke}. The estimates for local and 
global error follow 
standard steps and are presented next for completeness.  However, for {\em macroscopic observables} that typically arise 
in the simulation of extended KMC systems,
we prove estimates which are system-size independent in Section~\ref{macroscopic}. Finally, in
Section~\ref{infinite_volume}  we present, as a complementing  
theoretical perspective, the same estimates on the infinite lattice
$\Lambda=\Z^d$, in which case the involved generators are necessarily unbounded.

\begin{lemma}\label{RemainderBound}
Let $\LOPER$ be the generator of a strongly continuous contraction semigroup
$\{e^{t\LOPER}\}_{t\geq 0}$ on the Banach space $\CB$. Then the operators 
\begin{equation}
 \DOPER_m(t\LOPER) = \EXP{t\LOPER} - \sum_{k=0}^{m-1}\frac{t^k}{k!}\LOPER^k\COMMA 
     \;\;\;\;\; m\in\N^+
\end{equation}
satisfy the bound
\begin{equation}
 \NORM{\DOPER_m(t\LOPER)v} \leq \frac{t^k}{k!}\NORM{\LOPER^k v}, \SPACE \forall v \in\CB
\end{equation}
\end{lemma}

\begin{proof}
 {\rm
     see Jahnke, \cite{jahnke}.
 }
\end{proof}
\begin{lemma}[Local Error]\label{lem:local_error}
Let $\SEMIGRPLT$ and $\SEMIGRPST$ be the schemes \VIZ{lie} and \VIZ{strang}
associated with the Lie and Strang splittings respectively, and let $u(\DT) = \SEMIGRP(\DT) f$ be the solution
of \VIZ{ODE}. Then the local error for the Lie splitting is
\begin{equation}
  \NORM{  \SEMIGRPL(\DT)f - u(\DT)} \leq c_1\NORM{[\LOPA,\LOPB ]f}\DT^2 + c_2 \sum_{|m|=3} \NORM{\LOPA^{m_1}\LOPB^{m_2}f} \DT^3\COMMA
\end{equation}
and for the Strang splitting scheme
\begin{equation}
\begin{split}
   \NORM{ \SEMIGRPS(\DT)f - u(\DT) }  & \leq c_3\NORM{[\LOPA,[\LOPA,\LOPB]]f - 2[\LOPB,[\LOPB,\LOPA]]f}\DT^3 \\
                               & \quad + c_4 \sum_{|m|=4}    \NORM{\LOPA^{m_1}\LOPB^{m_2}\LOPA^{m_3}f}\DT^4
\end{split}
\end{equation}
where $[\LOPA,\LOPB]=\LOPA\LOPB - \LOPB\LOPA$ denotes the commutator of $\LOPA$ and $\LOPB$ and $c_i,\;i=1,...,4$ are
positive constants with $c_i < 1$.
\end{lemma}
\begin{proof}
Using Lemma~\ref{RemainderBound} the proof follows the standard finite dimensional approach based on the expansion
of the operator exponential. We present the calculations here for the sake of completeness.
In order to simplify the notation, we introduce
\begin{align}\label{Def:a_k}
a_{k,N}(\DT) = \begin{cases}
           \frac{\DT^k}{k!}\LOPA^k & \mathrm{if}\;\;k<N\COMMA \\
		  \DOPER_k(\DT \LOPA) & \mathrm{if}\;\;k=N>0\COMMA \\
		      \EXP{\DT \LOPA}  & \mathrm{if}\;\;k=N=0\COMMA
          \end{cases}
\; \; \;
b_{k,N}(\DT) = \begin{cases}
           \frac{\DT^k}{k!}\LOPB^k & \mathrm{if}\;\;k<N\COMMA \\
		  \DOPER_k(\DT \LOPB) & \mathrm{if}\;\;k=N>0\COMMA \\
		   \EXP{\DT \LOPB}  & \mathrm{if}\;\;k=N=0\PERIOD
          \end{cases}
\end{align}
Now the semigroup for the Lie splitting, at $t=\DT$, can be written as
\begin{eqnarray*}
 \EXP{\DT\LOPA}e^{\DT\LOPB}f &=& \sum_{i+j\leq 3}a_{i,3-j}(\DT)b_{j,3}(\DT)f \\
                             &=& \Bigl( I + \DT(\LOPA + \LOPB) + \frac{\DT^2}{2}(\LOPA + \LOPB)^2 \Bigr)f\\
                             &+& \DT^2 [\LOPA,\LOPB]f + \sum_{i+j = 3}a_{i,3-j}(\DT)b_{j,3}(\DT)f\PERIOD
\end{eqnarray*}
Comparing with
\begin{equation*}
 \EXP{\DT(\LOPA+\LOPB)}f = \Bigl( I + \DT(\LOPA + \LOPB) + \frac{\DT^2}{2}(\LOPA +\LOPB)^2 \Bigr)f + \DOPER_3(\DT(\LOPA+\LOPB))f\COMMA
\end{equation*}
we get the estimate for the local error
\begin{align*}
 &\NORM{\EXP{\DT\LOPA}\EXP{\DT\LOPB}f - \EXP{\DT(\LOPA+\LOPB)}f}  \leq \DT^2\NORM{ [\LOPA,\LOPB]f} \\ 
     & \qquad + \NORM{\DOPER_3(\DT(\LOPA+\LOPB))f} +\NORM{\sum_{i+j = 3}a_{i,3-j}(\DT)b_{j,3}(\DT)f}\PERIOD
\end{align*}
The second term in the above inequality is bounded by Lemma~\ref{RemainderBound} and the third term is bounded by
\begin{equation*}
  \NORM{\sum_{i+j = 3}a_{i,3-j}(\DT)b_{j,3}(\DT)f} \leq 
                       c\DT^3 \Bigl( \NORM{\LOPA^3f} +\NORM{\LOPA^2\LOPB f} + \NORM{\LOPA\LOPB^2f} + \NORM{\LOPB^3f} \Bigr)\COMMA
\end{equation*}
which follows from the definitions of $a_k$ and $b_k$. The last step completes the proof for the local error in the Lie case.
For the Strang scheme the proof follows the same idea, we only have to take one more term in the expansion,
\begin{align*}
  \EXP{\frac{\DT}{2}\LOPA}e^{\DT\LOPB}  \EXP{\frac{\DT}{2}\LOPA} f 
                                   & = \sum_{i+j+k\leq 4} a_{k,4-i-j}(\frac{\DT}{2})  b_{j,4-i}(\DT) a_{i,4}(\frac{\DT}{2})  f \\
                                   &= \Bigl( I + \DT(\LOPA + \LOPB) + \frac{\DT^2}{2}(\LOPA + \LOPB)^2 + 
                                      \frac{\DT^3}{6}(\LOPA + \LOPB)^3 \Bigr)f \\
                                   & \qquad + \frac{\DT^3}{24}\Bigl( [\LOPA,[\LOPA,\LOPB]] -
                                       2[\LOPB,[\LOPB,\LOPA]] \Bigr)f \\
                                   & \qquad  + \sum_{i+j+k = 4}
                                       a_{k,4-i-j}(\frac{\DT}{2}) b_{j,4-i}(\DT) a_{i,4}(\frac{\DT}{2}) f\PERIOD
\end{align*}
Comparing with
\begin{eqnarray*}
   \EXP{\DT(\LOPA+\LOPB)}f &=& \Bigl( I + \DT(\LOPA + \LOPB) + \frac{\DT^2}{2}(\LOPA +\LOPB)^2 \\ 
                           &+& \frac{\DT^3}{6}(\LOPA + \LOPB)^3 \Bigr)f + \DOPER_4(\DT(\LOPA+\LOPB))f\COMMA
\end{eqnarray*}
the estimate for the local error follows
\begin{align*}
 &\NORM{\EXP{\frac{\DT}{2}\LOPA}e^{\DT\LOPB}  \EXP{\frac{\DT}{2}\LOPA} f -
       \EXP{\DT(\LOPA+\LOPB)}f}  \leq c\DT^3 \NORM{[\LOPA,[\LOPA,\LOPB]]f -2[\LOPB,[\LOPB,\LOPA]]f} \\
           & \qquad + \NORM{\sum_{i+j+k = 4} a_{k,4-i-j}(\frac{\DT}{2}) b_{j,4-i}(\DT) a_{i,4}(\frac{\DT}{2}) f}
             + \NORM{\DOPER_4(\DT(\LOPA+\LOPB))f}\PERIOD
\end{align*}
The second term is bounded by Lemma~\ref{RemainderBound} and the third term is
bounded by
\begin{equation*}
 \NORM{\sum_{i+j+k = 4}
   a_{k,4-i-j}(\frac{\DT}{2}) b_{j,4-i}(\DT) a_{i,4}(\frac{\DT}{2}) f} \leq c_4\DT^4 \sum_{|m|=4}
  \NORM{\LOPA^{m_1}\LOPB^{m_2}\LOPA^{m_3}f}\COMMA
\end{equation*}
which again follows from \VIZ{Def:a_k}.
\end{proof}

After establishing the local truncation error it is straightforward to obtain the global error 
estimate.
\begin{theorem}[Global error]\label{error:global}
Let $\SEMIGRPLT$ and $\SEMIGRPST$ be the 
 the schemes \VIZ{lie} and \VIZ{strang}
associated with the Lie and Strang splittings respectively and let $u(t_n) = \SEMIGRP(t_n) f$ be the exact solution
of \VIZ{ODE}. Then the global error at the time $T=t_n = n h$, %
for the Lie splitting is bounded by
\begin{equation}\label{global_lie}
 \NORM{\SEMIGRPL(t_n) u(0) - u(t_n)} \leq C_1 \max_{k=0,\dots,n} \NORM{[\LOPA,\LOPB]u(t_k)} \DT + \mathcal{R}_L(u) \DT^2\COMMA
\end{equation}
where the remainder is given by
\begin{equation}\label{lie_remainder}
    \mathcal{R}_L(u) \equiv \mathcal{R}_L(u;n,\DT) =  C_2 \max_{k=0,\dots,n}\sum_{|m|=3} \NORM{\LOPA^{m_1}\LOPB^{m_2}u(t_k)}\PERIOD
\end{equation}
and for the Strang scheme
\begin{align}\label{global_strang}
 \NORM{\SEMIGRPS(t_n)u(0) - u(t_n)} \leq & C_3 \max_{k=0,\dots,n}\NORM{ \Bigl( [\LOPA,[\LOPA,\LOPB]] - 2[\LOPB,[\LOPB,\LOPA]]\Bigr)u(t_k)} \DT^2\\
                                          + \mathcal{R}_S(u) \DT^3\COMMA \nonumber
\end{align}
where
\begin{equation}\label{strang_remainder}
      \mathcal{R}_S(u) = \mathcal{R}_S(u;n,\DT) = C_4
\max_{k=0,\dots,n}\sum_{|m|=4}\NORM{\LOPA^{m_1}\LOPB^{m_2}\LOPA^{m_3}u(t_k)}\COMMA
\end{equation}
and $C_1,C_2,C_3$ and $C_4$ are constants, depending only on $T$.
\end{theorem}
\begin{proof}
It can be shown by induction that
$$
   e_n = \SEMIGRPB^n(\DT) u(0) - u(t_n) = \sum_{k=0}^{n-1} \SEMIGRPB^k(\DT) 
         \left({\SEMIGRPB(\DT) - \SEMIGRP(\DT)}\right) 
         \SEMIGRP^{(n-k-1)}(\DT) u(0)\PERIOD
$$
where $\SEMIGRPB$ denotes either $\SEMIGRPL$ or $\SEMIGRPS$.
By the assumptions, the operators $\LOPA$ and $\LOPB$ generate
strongly continuous contraction semigroups and thus $\NORM{\SEMIGRPB^k}\leq 1$, the global error is
bounded by
\begin{eqnarray*}
 \NORM{e_n} &\leq&   \sum_{k=0}^{n-1} \NORM{ \left({\SEMIGRPB(\DT) - \SEMIGRP(\DT)}\right) u(t_{n-k-1})} \\
            &\leq& n \max_{k=0,\ldots,n} \NORM{ \left({\SEMIGRPB(\DT) - \SEMIGRP(\DT)}\right)
                   u(t_{k})} \PERIOD
\end{eqnarray*}
Using Lemma~\ref{lem:local_error}, for $\SEMIGRPB=\SEMIGRPL$ and $\SEMIGRPB=\SEMIGRPS$, to estimate
the local error and the fact that $nh=T$ we obtain the estimates
\VIZ{global_lie} and \VIZ{global_strang} for the Lie and the Strang scheme respectively.

\end{proof}

\section{Estimates for macroscopic observables}\label{macroscopic}
In Theorem~\ref{error:global} we have shown that the proposed splitting schemes are convergent
as the time step $\DT$ tends to zero. 
The main idea of the proposed scheme is to control an error for observables, in other words we estimate
the weak error by analyzing solutions of \VIZ{ODE}.
If we restrict the initial data of the problem \VIZ{ODE} to a special class of functions, then
it is possible to show that the error terms are independent of the size of the lattice, $N$. It turns out that
this is a wide class of function containing some of the most common observables in KMC simulations,
such as mean coverage or spatial correlations, we refer to Section~\ref{section:examples_of_observables}.

In order to simplify the notation we suppress the dependence of the discrete derivative operator $\DERIV_{x,\omega}$ on
$\omega$ in Definition~\ref{definition:derivative}.

\begin{definition}\label{definition:high_derivatives}
For $\mathbf{x}=(x_1,\ldots,x_m)\in\LATT^m$ we introduce the notation
$$
\DERIV_{\mathbf{x}} f(\sigma)= \DERIV_{x_1}\ldots\DERIV_{x_m} f(\sigma)=\DERIV_{x_1\ldots
x_m}f(\sigma) \COMMA
$$
and we refer to it as the \textit{discrete derivative} of $f$ with respect to $\mathbf{x}$.
For example if $\mathbf{x}=(x,y)$ then
$$
   \DXY  f(\sigma) = \DX\DY f(\sigma) = f(\sigma^{xy}) - f(\sigma^x) -
                                        f(\sigma^y) + f(\sigma) \PERIOD
$$
\end{definition}

\begin{definition}
Let $\mathbf{x}=(x_1,\ldots,x_m)\in\LATT^m$ and $f\in\CBS$. Then we define the norm
$$
   \NORMA{m}{f} = \sum_{x_1\in\LATT}\ldots\sum_{x_m\in\LATT}
   \NORMA{\infty}{\DERIV_{\mathbf{x}}f} \COMMA
$$
and the function space
$$
   \CSS{m} = \{ f\in\CBS \SEP  \sum_{k=1}^m\NORMA{k}{f} \leq C_f \;\;\mathrm{where}\;\; C_f
             \;\;\mathrm{is\; independent \; of}\;\; N \} \COMMA \forall m \in\mathbb{N} \PERIOD
$$
\end{definition}
We refer to elements of $\CSS{m}$ as {\em macroscopic observables} and we will discuss examples
in Section \ref{section:examples_of_observables}. We now present the main theorem of this paper,  
showing that for such macroscopic observables, or equivalently  under  smoothness
conditions on the initial data, the global error estimates for the Lie and the Strang schemes are
independent of the dimension of the system. The proof of this theorem is contained in the next two
subsections.
\begin{theorem}\label{main}
 (a) Let $u(t)$ be the solution of \VIZ{ODE} with $u(0)= f \in \CSS{3}$.%
Then for the global error estimate of Theorem~\ref{error:global} on  the Lie scheme \VIZ{lie} we have
$$
  \NORM{\SEMIGRPL(t_n) u(0) - u(t_n)} \leq C_1 \max_{k=0,\dots,n} \NORM{[\LOPA,\LOPB]u(t_k)} \DT +
  \mathcal{R}_L(u) \DT^2\COMMA
$$
where
$$
  \NORM{[\LOPA,\LOPB]u(t_k)} < C %
$$
and 
$$ 
  \mathcal{R}_L(u) < \tilde C \COMMA
$$
where both constants $C$ and $\tilde C$ are independent of the system size $N$.
Moreover, if $u(0) \in \CSS{4}$ then for the global error of the Strang scheme
$$
  \NORM{ \Bigl( [\LOPA,[\LOPA,\LOPB]] - 2[\LOPB,[\LOPB,\LOPA]]\Bigr)u(t_k)} < C \COMMA
$$
and
$$ 
  \mathcal{R}_S(u) < \tilde C\COMMA 
$$ 
where the constants $C$ and $\tilde C$ are independent of the system size $N$.

\noindent
(b)  Many macroscopic observables $u(0)=f$ are not just in $\CSS{3}$ but also satisfy a local bound such as
\begin{equation}\label{localbound}
   \max_{z\in\LATT} \NORMINF{\DZ u(0,\cdotp)}+\max_{x, y\in\LATT} \NORM{\DXY  u(0,\cdot)} +\max_{x,
   y,z\in\LATT} \NORM{\DXYZ  u(0,\cdot)} \leq \frac{C}{N}\PERIOD
\end{equation}
Then the bounds for the commutators become
\begin{equation}\label{betterlie}
  \NORMINF{[\LOPA,\LOPB]u(t,\cdotp)} 
     \leq C\; \frac{L^{d+1}}{q} \COMMA
\end{equation}
and
\begin{equation}\label{betterstrang}
  \NORM{ \Bigl( [\LOPA,[\LOPA,\LOPB]] - 2[\LOPB,[\LOPB,\LOPA]]\Bigr)u(t_k)}     \leq C\; \frac{L^{2d+1}}{q} \COMMA
\end{equation}
where  $\tfrac{N}{M}=Q=q^d$ and the constant $C$ is independent of $N$. 
The parameters $L$, $M$, $N$, $q$ are defined in Definition~\ref{decomposition}, 
and $d$ is the dimension of the lattice $\LATT \subset \Lambda=\Z^d$.

\end{theorem}
The proof of this theorem is 
given in Section~\ref{proof_main}, while the supporting  results are proved earlier in Sections~\ref{remainder} and ~\ref{commutator}.
Next, we discuss typical examples of macroscopic observables $f$ which are used in KMC simulations
and  also satisfy the assumptions of Theorem~\ref{main}.

\subsection{Examples of observables}\label{section:examples_of_observables}

There is a wide class of macroscopic observable  functions in $\CSS{m}$, that satisfy
\begin{equation}\label{special_observables}
\DX f(\sigma) :=  
              \frac{1}{N}  \phi\bigl( \sigma(x+k_1), \cdots , \sigma(x+k_\ell) \bigr)\PERIOD
                    k_i\in\LATT \COMMA \forall x\in \LATT\COMMA
\end{equation}
A class of functions that satisfies
\VIZ{special_observables}, or more generally \VIZ{localbound}, includes the coverage, spatial correlations, Hamiltonians
and more generally observables of the type
$$
f(\sigma) = \frac{1}{N}\sum_{y\in\LATT} U\bigl(\sigma(y+k_1),  \ldots  ,\sigma(y+k_\ell)\bigr)
\COMMA k_i\in\LATT \PERIOD
$$
 These functions have the property that their
discrete derivatives depend only on a fixed number of points on the lattice that does not scale
with $N$.
In this section we will show that this class of function belong in $\CSS{m}$, $\forall m \in
\mathbb{N}^+$.

\begin{example}[Coverage]
{\rm 
Let $f(\sigma) = \bar \sigma =  \frac{1}{N}\sum_{x\in\LATT} \sigma(x)$, the observable that measures
the mean coverage of the lattice $\LATT$. Then
$$
   \DX f(\sigma) = \frac{1}{N}(\sigma^x(x)-\sigma(x))\COMMA
$$
and in the case $\sigma (x)\in\{0,1\}$ it takes the simple form
$$
   \DX f(\sigma) = \frac{1}{N}(1-2\sigma(x)) \PERIOD
$$
The {\it local average} over a percentage of the domain, defined as
$$f(\sigma) = \frac{1}{N}\sum_{x\in A\subset\LATT} \sigma(x)\COMMA$$ 
is also in the same class.
}
\end{example}
\begin{example}[Spatial correlations]
{\rm 
Let $f(\sigma ; k) = \frac{1}{N}\sum_{x\in\LATT} \sigma(x)\sigma(x+k)$, the mean spatial
correlation of length $k$. Then, when $\sigma (x)\in\{0,1\}$ it takes the form
$$
   \DX f(\sigma) = \frac{1}{N} \bigl(1-2\sigma(x) \bigr) \bigl(\sigma(x+k)+\sigma(x-k)\bigr)\PERIOD
$$
}
\end{example}

In these examples it is obvious that $f\in\CSS{1}$. To such functions we can apply
Lemma~\ref{derivatives:rates} and easily conclude that they belong to $\CSS{m}$ for
$m\leq m_0$, where $m_0$ depends on the form of the observable.

\begin{example}\label{derivatives:observable}
Let $f$ be an  observable of type \VIZ{special_observables} with $\ell=1$ and $k_1=0$,
then
$$
    \DX \DY f(\sigma) = \DX \frac{1}{N}\phi(\sigma(y)) = \frac{1}{N}\phi(\sigma^x(y)) -
                                       \frac{1}{N}\phi(\sigma(y)) = 0\COMMA\;\;\; |x-y|>1\PERIOD
$$
An analogous result holds when $\ell\geq 1$ and $k_i\neq0$ with $|x-y|>c(\ell)$, where the constant
depends on $\ell$ but not on $N$.
\end{example}

Finally, there are  macroscopic  observables that are not of the type \VIZ{special_observables} but  still  
satisfy \VIZ{localbound}
\begin{example}[Variance]
{\rm
Let $f(\sigma) = \frac{1}{N}\sum_{x\in\LATT}(\sigma(x)-\bar \sigma)^2 = \bar \sigma - {\bar\sigma}^2$. Then
$$
  \DX f(\sigma) =  \frac{1}{N}\bigl(1-2\sigma(x)\bigr)\Bigl( 1 - 2\bar\sigma + \frac{2\sigma(x)-1}{N} \Bigr) \PERIOD
$$
It is easy to verify that variance is in $\CSS{2}$ and satisfies \VIZ{localbound}.}
\end{example}

\subsection{Bounds on the remainder}\label{remainder}
In order to establish that the remainders $\mathcal{R}_L(u)$, \VIZ{lie_remainder}, 
or $\mathcal{R}_S(u)$, \VIZ{strang_remainder},  Theorem~\ref{error:global}, are bounded
by constants independent of $N$ we derive estimates for powers of the operators $\LOPA$, $\LOPB$
and their compositions such as  $\LOPA^2\LOPB$. The idea for such estimates is an easy extension of estimates 
on  $\LOPER^2$ acting on the solution of \VIZ{ODE}, which we present next.
First,  we prove that $\LOPER^2 u$ is bounded by the sum of first and second derivatives of $u$.
\begin{lemma}\label{bound:L2_from_derivatives}
 Let $u$ be the solution of equation  \VIZ{ODE}. Then for the operator $\LOPER^2$ the following bound holds
\begin{eqnarray}\label{L2_semibound}
   \NORMINF{\LOPER^2 u(t,\cdotp)} &\leq&  c_1 \sum_{x\in\LATT} \NORMINF{ \DX u(t,\cdotp)} +  c_2
   \sum_{x,y\in\LATT} \NORMINF{ \DXY u(t,\cdotp)} \nonumber \\
   &=& c_1 \NORMA{1}{u(t,\cdotp)} + c_2\NORMA{2}{u(t,\cdotp)} \PERIOD
\end{eqnarray}
\end{lemma}

\begin{proof}
By a straightforward calculation 
$$
   \LOPER^2 u(t,\sigma) = \sum_{x,y\in\LATT} c(x,\sigma)c(y,\sigma^x)\DXY u(t,\sigma) -
                          \sum_{x,y\in\LATT} c(x,\sigma)\DX c(y,\sigma)\DY u(t,\sigma)\COMMA
$$
and by taking norms on both sides
\begin{eqnarray*}
    \NORMINF{u(t,\cdotp)} &\leq&  \NORMINF{\sum_{x\in\LATT}\sum_{|x-y|\leq L} c(x,\cdotp) \DY
    u(t,\cdotp)} + c_2 \sum_{x,y\in\LATT} \NORMINF{\DXY u(t,\cdotp)} \\
    &\leq& c_1  \sum_{x\in\LATT} \NORMINF{\DY u(t,\cdotp)}
                             + c_2  \sum_{x,y\in\LATT} \NORMINF{\DXY u(t,\cdotp)} \COMMA
\end{eqnarray*}
where the first inequality follows from the fact that $\DX c(y,\sigma)=0$ when
$|x-y|>L$, see Lemma~\ref{derivatives:rates}, where we show that the derivatives of the rate functions have compact support
that depends only on the length of the interaction $L$.
\end{proof}

\begin{lemma}\label{derivatives:rates}
Let $c$ be a rate function with interactions of range $L$
$$
  c(a,\sigma) = \tilde c\bigl (\sigma(a-L),\dots,\sigma(a+L)\bigr ),\SPACE a\in\LATT \COMMA
$$
then 
$$
  \DX c(a,\sigma)=0, \;\;\;  \forall x\in\LATT \;\;\mathrm{with} \;\; |x-a|>L \COMMA 
$$ 
and
$$ 
\DERIV_{xy} c(a,\sigma)=0, \;\;\;  \forall x,y\in\LATT \;\;\mathrm{with} \;\; |x-y|>2L+1 \PERIOD 
$$
Moreover, for all higher derivatives holds that
$$
\DXK{1} \DXK{2} \ldots\ \DXK{n} f(\sigma) \equiv \prod_{k=1}^n \DXK{k} f(\sigma) = 0\COMMA\;\;\;\;
|x_i-x_j|>2L+1\COMMA\; i\neq j \PERIOD
$$ 
\end{lemma}

\begin{proof}
For the first discrete derivative it is sufficient to observe that if $x\neq y$ then
$\sigma^y(x)=\sigma(x)$. Thus when $a$ has distance from $x$ greater than $L$ the rate function
$c(a,\sigma)$ is equal to $c(a,\sigma^x)$ and the first derivative is zero.

For the second derivative, based on the calculation for the first derivative, we have
$$
  \DX \Bigl ( \DY c(a,\sigma) \Bigr ) = 0\COMMA\;\;\; |y-a|>L\COMMA
$$
or, if we interchange $x$ and $y$,
$$
  \DY \Bigl ( \DX c(a,\sigma) \Bigr ) = 0\COMMA\;\;\; |x-a|>L\PERIOD
$$
Finally,  the second derivative is always zero when $|x-y|>2L+1$.

For the general case, the proof follows from the fact that $\DX \DY  c(a,\sigma) = \DY \DX  c(a,\sigma)$ 
and from the following observation
$$
  \prod_{\substack{k=1 \\ k\neq i,j}}^n \DXK{k} \Bigl ( \DXK{i}\DXK{j}c(a,\sigma)\Bigr ) =0\COMMA\;\;\;
  |x_i-x_j|>2L+1\COMMA\;\;\; i\neq j\COMMA
$$
which is true by the result for the second derivative.
\end{proof}

\begin{proposition}\label{bound:L2_on_u}
Let $u(t,\sigma)$ be the solution of the equation \VIZ{ODE} with initial data in
$\CSS{2}$. Then the operator $\LOPER^2$ satisfies the bounds,
\begin{equation}
   \NORMINF{\LOPER^2 u(t,\cdotp)} \leq C\COMMA
\end{equation}
and
\begin{equation*}
  \NORMA{1}{u(t,\cdotp)} + \NORMA{2}{u(t,\cdotp)} \le C_1 \NORMA{1}{u(0,\cdotp)} + C_2 \NORMA{2}{u(0,\cdotp)} \COMMA
\end{equation*}
where $C, C_1$ and $C_2$ are constants independent of $N$. 
\end{proposition}

\begin{proof}
We will bound the right hand side of the equation \VIZ{L2_semibound} thus we need estimates on the
first and the second derivatives of $u$. For the sake of brevity we use a vectorial notation 
$\LOPER f = \mathbf{c}(\sigma)\cdot\nabla_\sigma f(\sigma) \equiv \sum_x c(x,\sigma) \DX f(\sigma)$.
The governing equations for $u$, $v_1\equiv\DX u$, $v_2\equiv\DX u$ and 
$w \equiv\DX\DY u$ are
\begin{eqnarray*}
 \partial_t u   &=& \mathbf{c}(\sigma) \cdot \nabla_\sigma u \\
 \partial_t v_1 &=& \mathbf{c}(\sigma) \cdot \nabla_\sigma v_1 +
            \DX\mathbf{c}(\sigma)\cdot\nabla_\sigma u(\sigma^x) \\
 \partial_t v_2 &=& \mathbf{c}(\sigma) \cdot \nabla_\sigma v_2 +
            \DY\mathbf{c}(\sigma)\cdot\nabla_\sigma u(\sigma^y) \\
 \partial_t w   &=& \mathbf{c}(\sigma) \cdot \nabla_\sigma w +
             \DY\mathbf{c}(\sigma)\cdot\nabla_\sigma v_1(\sigma^y) +
             \DX\mathbf{c}(\sigma)\cdot\nabla_\sigma v_2(\sigma^x) +
             \DXY \mathbf{c}(\sigma)\cdot\nabla_\sigma u(\sigma^{xy})\PERIOD
\end{eqnarray*}

First we bound the first derivative writing the solution for $v(t,\sigma)$
\begin{equation}\label{aux}
    \DX u(t,\sigma) = \EXP{\LOPER t} u(0,\sigma) + \int_0^t \EXP{(t-s)\LOPER} \sum_{|y-x|\leq N} \DX
c(y,\sigma) \DY u(s,\sigma^x)\, ds\PERIOD
\end{equation}
By taking the norms and summing over all $x\in\LATT$
\begin{eqnarray*}
     \sum_{x\in\LATT} \NORM{\DX u(t,\cdotp)} &\leq&  \sum_{x\in\LATT}
                      \NORM{\DX u(0,\cdotp)} + c_1 \int_0^t \sum_{x\in\LATT} \sum_{|y-x|\leq L}\NORM{\DY u(s,\cdotp)}\, ds\PERIOD
\end{eqnarray*}
By setting 
\begin{equation}\label{phi_norm}
  \varphi(t)  = \NORMA{1}{u(t,\cdotp)} = \sum_{x\in\LATT} \NORM{\DX u(t,\cdotp)} \COMMA
\end{equation}
we obtain
$$
 \varphi(t)  \leq  \varphi(0) + \bar c_1\int_0^t \varphi(s)\, ds\PERIOD
$$
Similarly, for the second derivatives we have, by using 
Lemma~\ref{derivatives:rates},
\begin{eqnarray*}
       \partial_t \DXY u(t,\sigma) = \LOPER \DXY u(t,\sigma) &+& \sum_{|z-y|\leq L}  \DY c(z,\sigma) 
                         \DXZ u(t,\sigma^y) + 
                         \sum_{|z-x|\leq L}  \DX c(z,\sigma) \partial_{yz} u(t,\sigma^x) \\
                         &+& \sum_{\substack{|z-x|\leq L \\ |z-y|\leq L}} \DXY c(z,\sigma) 
                         \DZ u(t,\sigma^{xy}) \chi_{C_{2L}}(x,y)\COMMA
\end{eqnarray*}
where $\chi_{C_{2L}}$ is the characteristic function and 
${C_{2L}}=\{ (x,y)\in\LATT^2\SEP |x-y|<2L \}$.
The solution of the above equation is expressed as 
\begin{eqnarray*}
\DXY u(t,\sigma) =\EXP{t \LOPER } \DXY u(0,\sigma) &+& \int_0^t \EXP{(t-s)\LOPER} \Biggl [\sum_{|z-x|\leq L} 
                \DY c(z,\sigma) \DXZ u(s,\sigma^y) + \sum_{|z-y|\leq L}  \DX c(z,\sigma)\partial_{yz} u(s,\sigma^x) \\
            &+& \sum_{\substack{|z-x|\leq L \\ |z-y|\leq L}} \DXY  c(z,\sigma) \partial_{z}
                u(s,\sigma^{xy}) \chi_{C_{2L}}(x,y)   \Biggr ]\, ds\PERIOD
\end{eqnarray*}
Thus, by using the contraction property of the semigroup and the fact that the discrete derivatives
of the rates are bounded functions, we have the estimate
\begin{equation}\label{2ndestimate}
\begin{aligned}
   \NORM{\DXY  u(t,\cdot)} \leq \NORM{\DXY u(0,\cdot)} 
       &+ c_1\int_0^t \sum_{|z-x|\leq L} \NORM{\DXZ u(s,\cdot)} ds\\
       &+ c_2\int_0^t \sum_{|z-y|\leq L} \NORM{\DYZ u(s,\cdot)}\,ds\\
       &+ c_3\int_0^t \sum_{\substack{|z-x|\leq L \\ |z-y|\leq L}}
      \NORM{\DZ u(s,\cdot)} \chi_{C_{2L}}(x,y)\, ds\PERIOD
\end{aligned}
\end{equation}
By summing over all $x,y\in\LATT$ and setting 
\begin{equation}\label{theta_norm}
 \vartheta(t)  = \NORMA{2}{u(t,\cdotp)} = \sum_{x,y\in\LATT} \NORM{\DXY u(t,\cdot)}\COMMA
\end{equation}
we obtain
$$
   \vartheta(t) \leq \vartheta(t) + \bar c_2 \int_0^t \vartheta(s) ds + \bar c_3 \int_0^t\varphi(s)\, ds\COMMA
$$
where both $\bar c_2$ and $\bar c_3$ depend on $L$ but not on $N$.
However, from Lemma~\ref{gronwall} (see Appendix) we have 
$$
    \varphi(t) \leq \tilde c_1 \varphi(0) = \tilde c_1 \NORMA{1}{u(0,\cdotp)} <C\COMMA
$$
where the last inequality follows from the assumption that $u(0,\sigma)\in \CSS{1}$.
Furthermore, from Lemma~\ref{gronwall} 
$$
   \vartheta(t) \leq \tilde c_2 \vartheta(0) + \tilde c_3 \varphi(0)\COMMA
$$
and the second term is bounded from the previous argument whereas the first term is equal to 
$$
   \vartheta(0) = \NORMA{2}{u(0,\cdotp)} < C\COMMA
$$
which is true because the initial data are in $\CSS{2}$.
Finally, we obtain  from Lemma~\ref{bound:L2_from_derivatives},
$$
\NORMINF{\LOPER^2 u(t,\cdotp)} \leq C_1\varphi(t) + C_2\vartheta(t) \leq C \PERIOD
$$
\end{proof}

\begin{remark}\label{bound:L2_on_u_easy}
{\rm
The same result can be obtained if we notice that
the function $v(t,\sigma)=\LOPER^2 u(t,\sigma)$ satisfies the equation \VIZ{ODE}. Then
the solution
can be written as $v(t,\sigma)=\EXP{t\LOPER^2} v(0,\sigma)$ and by taking the norm on both sides we
get the estimate
$$
  \NORMINF{\LOPER^2 u(t,\cdotp)} \leq \NORM{\EXP{t\LOPER^2}u(0,\cdotp)}  \leq \NORM{u(0,\cdotp)} \leq C\COMMA
$$
where the second inequality follows from the fact that
$\LOPER^2$ generates a contraction semigroup.
However, in order to get bounds for quantities like $\LOPA\LOPB u$, it is sufficient to observe from 
Lemma~\ref{bound:L2_from_derivatives} that
\begin{eqnarray*}
     \NORMINF{\LOPA\LOPB u(t,\cdotp)} &\leq& c_1 \sum_{\substack{x\in\LATT^1\\y\in\LATT^2}} 
               \NORMINF{\DXY u(t,\cdotp)} + c_2 \sum_{x\in\LATT^1} \NORMINF{\DX u(t,\cdotp)} \\ 
       &\leq& \NORMA{1}{u(t,\cdotp)} + \NORMA{2}{u(t,\cdotp)}
\end{eqnarray*}
and the norms on the right hand side are bounded from Proposition~\ref{bound:L2_on_u}. 
}
\end{remark}

Our last goal for this section is to prove that the remainders in the Lie and the Strang scheme, 
\VIZ{lie_remainder} and \VIZ{strang_remainder} respectively, are independent of the size of the
lattice. To achieve this, we first have to bound third and fourth powers of combinations of the operators $\LOPA$ and $\LOPB$
arising in \VIZ{lie_remainder} and \VIZ{strang_remainder}. Then, as
in Remark~\ref{bound:L2_on_u_easy}, using a more general form of
Lemma~\ref{bound:L2_from_derivatives} it is easy to prove that all relevant combinations of $\LOPA$ and $\LOPB$
are also bounded by constants independent of $N$.
We will present the general idea of the proof, rather than showing all the technical details that  anyhow
follow the same idea as in Proposition~\ref{bound:L2_on_u}. First we give a definition of the discrete
derivatives that generalizes Definition~\ref{definition:high_derivatives}. 

\begin{definition}
Let $\mathbf{x}=(x_1,\ldots,x_m)\in\LATT^m$ and  $\forall k \in \mathbb{N}\, , k \leq m$ define
$a\in\LATT^k$, the $k$-dimensional multi-index of $k$-tuples of $\mathbf{x}$. 
As in Definition~\ref{definition:derivative}, the discrete derivative with respect to $a$ is
$$
  \DERIV_{a} f(\sigma)= \DERIV_{a_1\ldots a_k} f(\sigma)\COMMA
$$
and we define
$\DERIV_{-a} f(\sigma)$
the derivative with respect to all variables in $\mathbf{x}=(x_1,\ldots,x_m)$ that are not contained in  $a$.
\end{definition}

Using this definition we are able to write a general form of the
governing equation for the $m$-th discrete derivative,
\begin{eqnarray}
    \partial_t \DERIV_\mathbf{x} u(t,\sigma) &=& \sideset{}{'}\sum_{0\leq|\alpha|\leq m}\DERIV_{\alpha} c(\sigma)
     \cdot \nabla_\sigma \DERIV_{-\alpha}  u(t,\sigma^{\alpha})\nonumber \\
    &=& \LOPER \DERIV_\mathbf{x} u(t,\sigma) + \sideset{}{'}\sum_{1\leq|\alpha|\leq m} \DERIV_{\alpha} c(\sigma)
    \cdot \nabla_\sigma \DERIV_{-\alpha} u(t,\sigma^{\alpha}) \COMMA
\end{eqnarray}
where the prime in the summation symbol means that we sum over tuples without distinguishing the
order of the variables, e.g., $\DXY=\DYX$. Using this representation and the
Remark~\ref{remark:gronwall} we can apply the same idea as in
Proposition~\ref{bound:L2_on_u} and prove bounds for the operator $\LOPER^k,k\in\mathbb{N}^+$,
with initial data in $\CSS{k}$.

\subsection{Bounds on the commutators}\label{commutator}
The constants in the local error estimate derived in Lemma~\ref{lem:local_error} involve bounds on
the commutators of the splitting
operators $\LOPER_1$ and $\LOPER_2$. We prove that these commutators are bounded operators on the
spaces $\CSS{m}$, independently of the system size $N$. 

The error analysis quantifies the intuitive link of the approximation error to the commutator
$[\LOPA,\LOPB]$ of the operators
$\LOPA$ and $\LOPB$. The commutator is directly related to the geometric decomposition and the range
of particle interactions. In order to demonstrate this relation more specifically we discuss an
example of Ising-type interacting system in which the events (updates) occur only at a single site $x\in\LATT$.

\smallskip
The error estimates in Lemma~\ref{lem:local_error} link the local error to the commutator of the
operators $\LOPER_1$ and $\LOPER_2$. 
In principle the commutator can be computed explicitly in terms of the rates $c(x,\omega;\sigma)$
although general formulae become too complicated
and impractical. Therefore we give an example for a specific example of single site events, i.e.,
$\omega = \{x\}$. The example also demonstrates
a procedure that is used for more involved cases. First we evaluate the commutators
associated with the decomposition of the lattice into
disjoint sub-lattices (Definition~\ref{decomposition}).

\begin{lemma}\label{lem:commutator}
Let $\LOPA,\LOPB$ be two operators defined by
\begin{equation*}
 \LOPA f(\sigma) = \ISINGOP{\CSET_1}{f}\COMMA\;\;\mbox{and}\;\;
 \LOPB f(\sigma) = \ISINGOP{\CSET_2}{f}\COMMA
\end{equation*}
and $\CSET_1,\CSET_2 \subset\LATT$ with $\mathrm{dist}(\CSET_1,\CSET_2)> L$. Then $\LOPA$
and $\LOPB$ commute, i.e.,
\begin{equation*}
 [\LOPA,\LOPB] = 0\PERIOD
\end{equation*}
\end{lemma}

\begin{proof} The proof follows from the straightforward calculation based on the fact that
$c(x,\sigma^y)=c(x,\sigma)$ when $x\in
\CSET_1$ and $y\in \CSET_2$ or vice versa and $f(\sigma^{xy})=f(\sigma^{yx})$. 
By a direct calculation we get
\begin{align*}
  \LOPA\LOPB f(\sigma) &= \sum_{x\in \CSET_1} c(x,\sigma) \Bigl[ \LOPB f(\sigma^x)-\LOPB f(\sigma)\Bigr] \\
    &= \sum_{x\in \CSET_1}\sum_{y\in \CSET_2} c(x,\sigma)c(y,\sigma)\Bigl(f(\sigma^{xy})-f(\sigma^x)-f(\sigma^y)+f(\sigma) \Bigr) \\
    &= \sum_{y\in \CSET_2}c(y,\sigma) \sum_{x\in \CSET_1} c(x,\sigma)\Bigl(f(\sigma^{xy})-f(\sigma^x)-f(\sigma^y)+f(\sigma) \Bigr) \\
    &= \sum_{y\in \CSET_2}c(y,\sigma) \Bigl(  \sum_{x\in\CSET_1}c(x,\sigma^y)[f(\sigma^{yx})-f(\sigma^y)] - \sum_{x\in \CSET_1}
        c(x,\sigma)[f(\sigma^x)-f(\sigma)] \Bigr)\\
    &= \sum_{y\in \CSET_2} c(y,\sigma) \Bigl[ \LOPA f(\sigma^y) -\LOPA f(\sigma) \Bigr] \\
    &=  \LOPB\LOPA f(\sigma)\PERIOD
\end{align*}
\end{proof}

\begin{lemma}\label{prop:LieBracketLie}
Let $\CSET_1$ and $\CSET_2$ be such that $\CSET_i = \CSET_i^o \cup \CSET_i^\partial$, where
$\CSET_i^o := \{ x\in\CSET_i \: |  \: \DIST(x,(\CSET_i)^c) >  L \}$, where $A^c$ is the
complement of set $A$. With further decomposition $\CSET_i^o = \CSET^{oo}_i + \CSET^{o\partial}_i$
where 
$\CSET_i^{oo} := \{ x\in\CSET_i \: |  \: \DIST(x,(\CSET_i)^c) >  2L \}$ (see
Figure~\ref{fig:sublattice}).
Let $\LOPER_i = \LOPER_i^o+\LOPER_i^{\partial}$ and $\LOPER_i^o = \LOPER_i^{oo}
+ \LOPER_i^{o\partial},\; i=1,2$ be the corresponding decomposition of the generator $\LOPER$,
then
\begin{equation*}
 [\LOPA,\LOPB] = [\LOPA^{\partial},\LOPB^{\partial}] 
\end{equation*}
and
\begin{equation*}
 [\LOPA,[\LOPA,\LOPB]]  =  [[\LOPA^{o\partial},\LOPA^\partial],\LOPB^\partial]
      + [\LOPA^\partial, [\LOPA^\partial,\LOPB^\partial]] \PERIOD
\end{equation*}
\end{lemma}

\begin{proof} The proof of the first statement follows directly from Lemma~\ref{lem:commutator} by
observing that
$\DIST(\CSET_1^o,\CSET_2^o) = 2L$ and
$$
   \DIST(\CSET_1^o,\CSET_2^\partial) = \mathrm{dist}(\CSET_1^\partial,\CSET_2^o) = L\PERIOD
$$
For the second statement, using the same lemma, we compute
\begin{equation*}
   [\LOPA,[\LOPA,\LOPB]]   =  [\LOPA^o,\LOPA^\partial\LOPB^\partial] -
        [\LOPA^o,\LOPB^\partial\LOPA^\partial] + [\LOPA^\partial,[\LOPA^\partial,\LOPB^\partial]]\PERIOD
\end{equation*}
The first term on the right hand side can be further simplified
\begin{align*}
 [\LOPA^o,\LOPA^\partial\LOPB^\partial] &=  \LOPA^o\LOPA^\partial\LOPB^\partial-
\LOPA^\partial\LOPB^\partial\LOPA^o 
                                         = \LOPA^o\LOPA^\partial\LOPB^\partial -
\LOPA^\partial\LOPA^o\LOPB^\partial \\
                                        &= [\LOPA^o,\LOPA^\partial]\LOPB^\partial
                                         = [\LOPA^{oo} +
\LOPA^{o\partial},\LOPA^\partial]\LOPB^\partial \\
                                        &= [\LOPA^{o\partial},\LOPA^\partial]\LOPB^\partial\COMMA
\end{align*}
where in the second equation we used the fact that $\LOPB^\partial\LOPA^o =\LOPA^o\LOPB^\partial $
and in the last equation
$[\LOPA^{oo},\LOPA^\partial]=0$. The same procedure leads to simplifying
the second term but the third cannot be simplified further. Combining all these steps
we obtain the result of the proposition.
\end{proof}

\begin{figure}[ht]
  \centerline{%
        \includegraphics[scale=0.4]{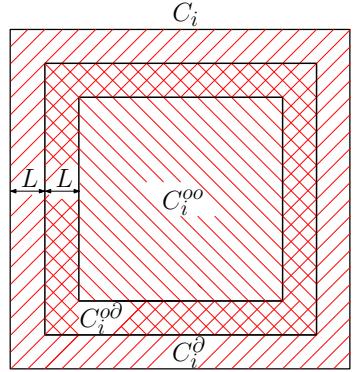}
  }%
  \caption{\label{fig:sublattice} Sub-lattice partitioning. Note that we use the notation
           $\partial C_i$ to denote $\bar C \setminus C$, see also Figure~\ref{fig:lattice_partition}.}
\end{figure}

The estimation of  the commutator in Theorem~\ref{main} requires {\em local estimates} on the first and
second discrete derivatives of the solution to the backward Kolmogorov equation by the discrete derivatives of the
initial data
\begin{lemma}\label{bound:derivative_of_u}
 The solution of the equation
\begin{equation}
   \partial_t u =\LOPER u\COMMA\;\;\; t \in (0,T]\COMMA\;\;\;\; u(0,\sigma)=f(\sigma)\COMMA
\end{equation}
 satisfies the bounds
\begin{equation}\label{1stlocal}
 \max_{x\in\LATT} \NORMINF{\DX u(t,\cdotp)} \leq C   \max_{x\in\LATT} \NORMINF{\DX u(0,\cdotp)}
\end{equation}
and
\begin{equation}\label{2ndlocal}
  \max_{x, y\in\LATT} \NORM{\DXY  u(t,\cdot)} \leq C   \big[
  \max_{x\in\LATT} \NORMINF{\DX u(0,\cdotp)}+\max_{x, y\in\LATT} \NORM{\DXY  u(0,\cdot)} \big]\COMMA
\end{equation}
where $C$ is a constant independent of $N$, however, it may depend exponentially on $t$.
\end{lemma}

\begin{proof}
 Using \VIZ{aux} and Lemma~\ref{derivatives:rates}, we have
\begin{eqnarray}\label{1stbound}
   \NORMINF{\DX u(t,\cdotp)}  \leq \NORMINF{\DX u(0,\cdotp)} &+& \BIGO(1) \int_0^t 
   \NORMINF{\DX u(s,\cdotp)} \, ds \nonumber\\ 
                 &+&  \BIGO(\frac{1}{L}) \int_0^t \sum_{|x-y|\leq L} \NORMINF{\DY u(s,\cdotp)} \, ds \PERIOD
\end{eqnarray}
Here the symbol $\BIGO$ is asymptotic in the size of the system $N\to\infty$. 
Setting $\gamma(t) = \max_{x\in\LATT} \NORMINF{\DX u(t,\cdotp)}$ we have
\begin{eqnarray*}
   \NORMINF{\DX u(t,\cdotp)}  \leq \gamma(0) + \BIGO(1) \int_0^t\gamma(s)ds \nonumber +  
                                                      \BIGO(\frac{1}{L}) L \int_0^t \gamma(s) \,ds \COMMA
\end{eqnarray*}
or 
$$
  \gamma(t)  \leq \gamma(0) + \BIGO(1) \int_0^t \gamma(s)\, ds \PERIOD
$$
Applying Gronwall's inequality we conclude the proof and obtain the bound
\begin{equation*}
  \gamma(t)  \leq \EXP{ct} \gamma(0) \PERIOD 
\end{equation*}
The inequality \VIZ{2ndlocal} follows similarly from \VIZ{2ndestimate} and from Lemma~\ref{gronwall}.
\end{proof}

The commutator, as shown in Lemma~\ref{prop:LieBracketLie}, is a localized quantity that
depends only on the boundary sites of the decomposed sub-lattices. Thus the  localized
estimate in Lemma~\ref{bound:derivative_of_u} gives us  a tool in order to reveal the scaling of
the commutator when acting on macroscopic observables.%

\subsection{Proof of Theorem~\ref{main}}\label{proof_main}
By  Lemma~\ref{prop:LieBracketLie}, the commutator can be written as
$[\LOPA^\partial,\LOPB^\partial]$, which due to Lemma~\ref{lem:commutator}  is expanded to
\begin{eqnarray*}
  [\LOPA^\partial,\LOPB^\partial]u(t,\sigma) =
\sum_{\substack{ x\in\Lambda_1^\partial,y\in\Lambda_2^\partial \\ |x-y|\le L}}
                       && c_1(x,\sigma)c_2(y,\sigma^x)\DY u(\sigma^x,t)  -
                            c_1(x,\sigma)c_2(y,\sigma)\DY u(\sigma,t)  \\
                      &-& c_1(x,\sigma^y)c_2(y,\sigma)\DX u(\sigma^x,t)
                            + c_1(x,\sigma)c_2(y,\sigma)\DX u(\sigma,t) \PERIOD
\end{eqnarray*}
On the other hand, by a straightforward calculation, we have
\begin{eqnarray*}
   \LOPA^\partial\LOPB^\partial u(t,\sigma) =
  && \sum_{\substack{x\in\Lambda_1^\partial,y\in\Lambda_2^\partial \\ |x-y|\le L}}
      c_1(x,\sigma)c_2(y,\sigma^x)\DXY u(t,\sigma) \\
  &&-\sum_{\substack{ x\in\Lambda_1^\partial,y\in\Lambda_2^\partial \\|x-y|\le L}} 
      c_1(x,\sigma)\DX c_2(y,\sigma)\DY u(t,\sigma)\PERIOD
\end{eqnarray*}                          
Taking norms on both sides similarly to Lemma~\ref{bound:L2_from_derivatives} and using the fact that the rates are bounded functions on
$\LATT\times\Sigma$, 
\begin{eqnarray}\label{bound:commutator_max}
  \NORMINF{[\LOPA,\LOPB]u(t,\cdotp)} 
   &\leq& C \sum_{\substack{ x\in\Lambda_1^\partial,y\in\Lambda_2^\partial \\ |x-y|\le L}} 
     \NORMINF{\DXY u(t,\cdotp)} + \NORMINF{\DY u(t,\cdotp)} \leq C \COMMA
\end{eqnarray}
where the second inequality follows from Proposition~\ref{bound:L2_on_u}, using the
fact that the initial data are \textit{macroscopic observables}, i.e., belong to $\CSS{2}$.
Similarly, we obtain the commutator estimate for the Strang scheme.

Next, we turn our attention to \VIZ{betterlie}. Many observables are in $\CSS{2}$,  
but also satisfy the local bound \VIZ{localbound} as one can see in Section~\ref{section:examples_of_observables}. 
Under this assumption, we obtain  from 
\VIZ{bound:commutator_max} the bound for the commutator

\begin{multline}\label{bound:commutator_max2}
  \NORMINF{[\LOPA,\LOPB]u(t,\cdotp)}  \leq C
         \sum_{ \substack{ x\in\Lambda_1^\partial,y\in\Lambda_2^\partial \\ |x-y|\le L}}  \NORMINF{\DXY
                 u(t,\cdotp)}+\NORMINF{\DY u(t,\cdotp)}  \\
                 \leq C \big[ \max_{x, y\in\LATT} \NORMINF{\DXY u(0,\cdotp)}+\max_{y\in\LATT}
                 \NORMINF{\DY u(0,\cdotp)}\big] 
                  \sum_{\substack{ x\in\Lambda_1^\partial,y\in\Lambda_2^\partial \\ |x-y|\le L}}1\COMMA
\end{multline}
where the second inequality follows from Lemma~\ref{bound:derivative_of_u}. Using the fact that
the initial data  belong to $\CSS{2}$ and satisfy \VIZ{localbound}, as well as 
that $|C_m^\partial|=c(d)Lq^{d-1}$, where $d$ is the dimension, we deduce that
\begin{equation}\label{betterlie2}
  \NORMINF{[\LOPA,\LOPB]u(t,\cdotp)} \leq {\frac{\tilde C}{N}} 
   \sum_{ \substack{ x\in\Lambda_1^\partial,y\in\Lambda_2^\partial \\ |x-y|\le L}} 1
     \leq \frac{\tilde C}{N} \times M \times c(d) Lq^{d-1}\times L^d= C\frac{L^{d+1}}{q}\COMMA
\end{equation}
where we used the fact that $\frac{N}{M}=Q=q^d$. We note that for more general, non-square lattices, the estimate is modified
accordingly as the structure of neighbors in the calculation of $|C_m^\partial|$ will evidently change.
Finally, the proof of \VIZ{betterstrang} follows along the same lines, noting that the the summation in \VIZ{betterlie2} 
is now replaced by summations such as 
\begin{equation*}
       \sum_{\substack{x\in\Lambda_1^\partial,y\in\Lambda_2^\partial, z\in \Lambda_1^\partial \\|x-y|\le L, |x-z|\le L}} 1\leq 
       M \times c(d) Lq^{d-1}\times L^d \times L^d \PERIOD
\end{equation*}
\section{Processor communication and error analysis}\label{PCS}
In this Section we examine the  balance between accuracy  and processor communication in the parallel Fractional 
Step KMC algorithms. Our analysis is based on the local and global error analysis tools we have developed in this article.

A key feature of the fractional step methods is what we define as  the  Processor Communication Schedule (PCS), which dictates 
the order with which the hierarchy of operators in \VIZ{opdecomp} are applied and for how long.  
For instance,  for the Lie scheme \VIZ{lie} the  processors corresponding 
to $\LOPA$ (resp. $\LOPB$) do not communicate, hence the processor communication within the
algorithm occurs {\em only} each time we have to apply   $\EXP{\Delta t\LOPA}$ or
 $\EXP{\Delta t\LOPB}$.  For this reason, we characterize the FS-KMC algorithms \VIZ{lie}, \VIZ{strang} as 
{\em partially asynchronous} since there is no processor communication during 
the period $\Delta t$. Furthermore, at every $\Delta t$
we have only local synchronization between processors, i.e., between the sets
$\CLSR\CUBE_m\cap\CLSR\CUBE_{m'}$ when $m\in\INDXE$ and $m'\in\INDXO$.  Hence, 
the bigger the allowable $\Delta t$ in \VIZ{lie} or in \VIZ{strang} the less processor communication we have, in which case 
the error in the approximation \VIZ{lie} or \VIZ{strang} worsens.

In both schemes \VIZ{lie}, and  \VIZ{strang}, the communication schedule is fully deterministic, 
relying on the Trotter Theorem. %
On the other hand, we can construct general randomized PCS based on the  {\em Random Trotter Product} Theorem, \cite{Kurtz}. 
Indeed,   the  sub-lattice parallelization algorithm for KMC,
introduced in \cite{ShimAmar05b},  is a particular example of a fractional step algorithm with stochastic PCS.
In \cite{ShimAmar05b, SPPARKS}
each sub-lattice is selected at random, independently  and  advanced by KMC over a fixed time window
$\Delta t$, subsequently  a new random selection is made and again the sub-lattice is advanced by $\Delta t$, etc. 
This algorithm is easily recast as a fractional step approximation, \cite{AKPTX}. 

Here we compare the deterministic and randomized PCS from the point of view of processor
communication and error analysis: we specify the same error tolerance TOL
for all PCS, which by means of our error analysis selects in each case a possibly different time windows $\Delta t$. 
Larger time windows $\Delta t$ give rise to algorithms that have 
less processor communication for the same error tolerance. %

\subsection{Randomized processor communication schedules}
A generalization by Kurtz, \cite{Kurtz},
of the Trotter Theorem suggests numerically  consistent schemes in which evolutions are applied not in a deterministic,
prescribed, order but as a random composition of individual propagators resulting in a random evolution. 
Given a pure jump process $X(t)$, with stationary measure $\mu(d\xi)$, and given 
the infinitesimal generators $\LOPER_k$ we define  a random evolution by 
$$
 \mathcal{T}_n(t) f = \EXP{\tau_0/n\LOPER_{\xi_0}}\EXP{\tau_1/n\LOPER_{\xi_1}}\dots\EXP{\tau_{N(nt)}/n\LOPER_{\xi_{N(nt)}}} f\COMMA
$$
where $N(t)$ is the number of jumps up to time $t$ and $\tau_k$ are the sojourn (waiting) times at the visited states 
$(\xi_0,\dots,\xi_{N(t)})$. The random Trotter product theorem yields the expectation semigroup
\begin{equation}\label{randomTrotter}
 \lim_{n\to \infty} \mathcal{T}_n(t) f = \EXP{t\bar\LOPER}f\COMMA\;\;\;\mbox{a.s.}
\end{equation}
with the generator $\bar\LOPER$ characterized explicitly 
\begin{equation}\label{averaged_generator}
  \bar\LOPER f = \int \LOPER_{\xi}f \,\mu(d\xi)\PERIOD
\end{equation}
While the random Trotter formula serves as a motivation for constructing schemes in which the evolution
of the system, i.e., the process $\PROCMIC$, is {\it approximated} by a process obtained from a random
composition of propagators $\EXP{\Delta t\LOPER_k}$,  the error analysis in the spirit of Theorem~\ref{main} 
requires more careful inspection of the
approximating process $\PROCAPPRGAM$ on the interval $[0,T]$ with $T = n\Delta t$. 

We present the construction in a simpler case of the independent identically distributed random variables that
index the individual generators $\LOPER_\xi$. We analyze the randomized Lie scheme for the operator splitting
given by $\LOPER = \LOPER_1 + \LOPER_2$.  In the context of the parallel FS-KMC
the random process $X(t)$ can be interpreted as a {\em stochastic} PCS. In \cite{AKPTX} we demonstrated that the  sub-lattice 
parallelization algorithm for KMC,
introduced in \cite{ShimAmar05b},  is a particular example of a fractional step algorithm with stochastic PCS.
In \cite{ShimAmar05b}
each sub-lattice is selected at random, independently  and  advanced by KMC over a fixed time window
$\Delta t=\DT$,
subsequently  a new random selection is made and again the sub-lattice is advanced by $\DT$, etc. 
This algorithm is easily recast as a fractional step approximation, where we can show that %
$\bar\LOPER =\frac{1}{2} \left(\LOPA+\LOPB\right)$
which is a time-rescaling of the original operator $\LOPER$. 
From the numerical analysis viewpoint, our   re-interpretation  of the  algorithm in  \cite{ShimAmar05b} as \VIZ{randomTrotter}  allows us 
to provide a   rigorous justification  that it is a {\em consistent} estimator of the serial KMC algorithm. 
Next we present the local error analysis of randomized PCS
and in analogy to Lemma~\ref{lem:local_error}, we estimate the mean (weak) local  error of the approximating $\gamma$-process. 

\begin{definition}[Random Lie splitting]\label{Lieapproxprocess}
 Let $\SEMIGRP_i(t)$, $i=1,2$, be two Markov semigroups with the infinitesimal generators $\LOPER_i$
 and the transition probability kernels $p_i(t;\gamma,\gamma')$.
 Assume $\{\xi_1,\xi_2,\dots\}$ be a sequence of i.i.d. Bernoulli random variables with values $\xi\in\{1,2\}$.
 We define the random evolution as the process $\PROCAPPRGAM$ by setting for $h>0$, $k=0,1,2,\dots, n$, and $\xi_{2k}$,
 $\xi_{2k-1}$ independent of $\gamma_0,\gamma_h,\dots,\gamma_{(k-1)h}$
 \begin{equation}\label{gammaproc}
    \EXPECT[f(\gamma_{kh})\SEP \gamma_{(k-1)h}] := \SEMIGRP_{\xi_{2k-1}}(h)\SEMIGRP_{\xi_{2k}}(h) f(\gamma_{(2k-1)h})\COMMA
 \end{equation}
 where the transition probability kernel is 
 $$
   [\SEMIGRP_{\xi_1}(h)\SEMIGRP_{\xi_2}(h)f](\eta) = 
     \sum_{\gamma'} \sum_{\gamma''} p_{\xi_1}(h;\eta,\gamma') p_{\xi_2}(h;\gamma',\gamma'') f(\gamma'')\PERIOD
 $$
\end{definition}

For a given $f\in\CBS$
we estimate the quantity $\EXPECT^\sigma[f(\sigma_{kh})]$ and $\EXPECT^\gamma[f(\gamma_{kh})]$ where the 
expected values are computed on the corresponding probability spaces associated with each process and
conditioned on the initial states $\sigma_0=\sigma$ and $\gamma_0 = \gamma$ respectively. We 
denote the initial states by different letters in order to distinguish between these two different probability path
measures, however, the initial state is assumed to be same for both $\PROCMIC$ and $\PROCAPPRGAM$.
\begin{theorem}[Local Error]\label{random_local_error}
 Assume 
 $\PROB{\xi_k=1}=\PROB{\xi_k=2}=\frac{1}{2}$, 
 for  the approximating process $\PROCAPPRGAM$ of Definition~\ref{Lieapproxprocess}. Then for  any $f\in \CBS$ and given  
 $\Delta t=h>0$,  the exact
 process $\PROCMIC$ with $\sigma_0 = \gamma_0=\gamma$ corresponding to the generator $\frac{1}{2}\LOPER$ satisfies
\begin{eqnarray*}\label{localmeanerror2}
  \EXPECT^\gamma[f(\gamma_h)] - \EXPECT^\sigma[f(\sigma_h)] &=& \EXPECT^\xi\left[\left(\SEMIGRP_{\xi_1}(h)\SEMIGRP_{\xi_{2}}(h)f(\gamma) -
    u(\gamma, \DT)\right)\right] \\
    &=& 
    \frac{h^2}{2} \EXPECT^\xi\left[ \LOPER^2_{\xi_1} + \LOPER^2_{\xi_2} + 2\LOPER_{\xi_1}\LOPER_{\xi_2}
   -\frac{1}{4}\LOPER^2\right]f(\gamma) +  \BIGO(h^3)\PERIOD   
\end{eqnarray*}
where $u(\gamma, \DT) = \SEMIGRP(\DT) f(\gamma)$ is the solution
of the rescaled, by $1/2$, equation \VIZ{ODE}
\begin{equation}\label{ODE2}
  \partial_t u(\zeta, t)=\frac{1}{2}\LOPER u(\zeta, t)\, , \quad\quad u(\zeta, 0)=f(\zeta)\PERIOD
\end{equation} 
\end{theorem}

\begin{proof}
We  estimate the local truncation error following similar steps as in the deterministic case.
From the definition of the $\gamma$-process we have
$$
 \EXPECT^\gamma[f(\gamma_{h})] = \EXPECT^\xi[\SEMIGRP_{\xi_{1}}(h)\SEMIGRP_{\xi_{2}}(h) f(\gamma)]\COMMA
$$
and similarly, using the fact that the initial states are same, $\sigma_0 = \gamma_0 = \gamma$, 
$$
 \EXPECT^\sigma[f(\sigma_{h})] = \SEMIGRP(h) f(\gamma)=u(\gamma, \DT)\PERIOD
$$
Hence we obtain a representation of the mean local error
\begin{equation}\label{meanerror}
  \EXPECT^\gamma[f(\gamma_h)] - \EXPECT^\sigma[f(\sigma_h)] = \EXPECT^\xi\left[\left(\SEMIGRP_{\xi_1}(h)\SEMIGRP_{\xi_{2}}(h) -
    \SEMIGRP(h)\right)f(\gamma)\right]\PERIOD
\end{equation}
Now for given realizations of $\xi_1$, $\xi_2$ we have the expansion of 
$\SEMIGRP_{\xi_1}(h)\SEMIGRP_{\xi_{2}}(h) - \SEMIGRP(h)$ as in the deterministic case, thus obtaining
\begin{equation}\label{errorrep}
\begin{aligned} 
 & [\SEMIGRP_{\xi_1}(h)\SEMIGRP_{\xi_{2}}(h) - \SEMIGRP(h)]f = \\ 
 & h[\LOPER_{\xi_1} + \LOPER_{\xi_2} - \frac{1}{2}\LOPER] f + \frac{h^2}{2}
      [\LOPER^2_{\xi_1} + \LOPER^2_{\xi_2} + 2\LOPER_{\xi_1}\LOPER_{\xi_2} - \frac{1}{4}\LOPER^2] f + \BIGO(h^3)\PERIOD
\end{aligned}
\end{equation}
Note that 
$\frac{1}{2}\LOPER=\frac{1}{2}\LOPER_1+\frac{1}{2}\LOPER_2$ is associated with  the process
$\PROCMIC$.
We have that the leading term of the local truncation error is $\EXPECT^\xi[\LOPER_{\xi_1} + \LOPER_{\xi_2} - \frac{1}{2}\LOPER]$ 
and thus this term vanishes whenever
$\frac{1}{2}\LOPER = \EXPECT^\xi[\LOPER_{\xi_1} + \LOPER_{\xi_2}]$, which holds true when
$P(\xi_k=1)=P(\xi_k=2)=\frac{1}{2}$. 
\end{proof}

\begin{remark}
{\rm
For  Bernoulli variables $\xi_i$ with probabilities $p$ this means that the $\gamma$-process
approximates a process with the generator $p\LOPER_1 + (1-p)\LOPER_2$ instead of $\frac{1}{2}
\LOPER=\frac{1}{2} \LOPER_1 + \frac{1}{2} \LOPER_2$. 
This indicates that the usual order
of the Lie splitting is achieved by properly weighing the time steps, i.e., applying $\SEMIGRP_1(h_1)$ and $\SEMIGRP_2(h_2)$ 
with different time steps $h_1$ and $h_2$ respectively. This calculation also shows that if we want to obtain the
generator $\LOPER$ instead of $\frac{1}{2}\LOPER$ in Lemma~\ref{random_local_error}, then in order  to evolve the process 
$\sigma$ by the time step $h$,  each semigroup $\SEMIGRP_{\xi_1}, \SEMIGRP_{\xi_{2}}$ needs to be 
applied with the time step $2h$, giving rise to the approximating process $\gamma_h$. In this case we have the local error representation
\begin{eqnarray}\label{localmeanerror3}
  \EXPECT^\gamma[f(\gamma_h)] - \EXPECT^\sigma[f(\sigma_h)] :&=& \EXPECT^\xi\left[\left(\SEMIGRP_{\xi_1}(2h)\SEMIGRP_{\xi_{2}}(2h)f(\gamma) -
     u(\gamma, \DT)\right)\right] \nonumber \\
    &=& 
 \frac{h^2}{2} \EXPECT^\xi\left[ 4\LOPER^2_{\xi_1} + 4\LOPER^2_{\xi_2} + 8\LOPER_{\xi_1}\LOPER_{\xi_2} -\LOPER^2\right]f(\gamma)\\
 &&+  \BIGO(h^3)   \nonumber \COMMA 
\end{eqnarray}
where $u(\gamma, \DT) = \SEMIGRP(\DT) f(\gamma)$ is the solution of \VIZ{ODE}.
}
\end{remark}

\subsection{Comparison of deterministic and random schedules}
The presented error analysis allows us to evaluate and compare   deterministic (Lie and Strang) PCS introduced in \cite{AKPTX}, 
as well as randomized PCS such as the one in Lemma~\ref{localmeanerror2},
introduced earlier in  \cite{ShimAmar05b}. We compare the deterministic and randomized PCS from the
point of view of processor communication and error analysis by specifying the same error tolerance TOL
for all PCS which, by means of our error analysis, selects in each case a possibly different time window $\Delta t$. 
Larger time windows  give rise to algorithms that have 
less processor communication for the same error tolerance. We start with the Lie and Strang schemes. 

We fix the same error  tolerance level TOL in the Lie and Strang global errors
\VIZ{global_lie} and \VIZ{global_strang} respectively.
We also fix  the same time window $T=n_L \Delta t_L$ and $T=n_S\Delta_S$ where $\Delta t_L$ and $\Delta t_S$ are 
the respective time steps of the Lie and the Strang schemes that will ensure  
the same tolerance level TOL up to time $T$.  Based on Theorems~\ref{error:global} and~\ref{main}
we have that the leading errors are governed by the commutators
\begin{equation}\label{error_communication_lie}
    \mathrm{TOL}  \sim C_{\mathrm{Lie}}(T) \Delta t_{\mathrm{Lie}}\COMMA \;\quad
    C_{\mathrm{Lie}}(T)=\max_{k=0,\dots,n} \NORM{[\LOPA,\LOPB]u(t_k)}\COMMA
\end{equation}
and
\begin{eqnarray}\label{error_communication_strang}
  \mathrm{TOL}  &\sim& C_{\mathrm{Strang}}(T) \Delta t_{\mathrm{Strang}}^2\, , \\\quad C_{\mathrm{Strang}}(T)&=&
  \max_{k=0,\dots,n}\NORM{ \Bigl( [\LOPA,[\LOPA,\LOPB]] - 2[\LOPB,[\LOPB,\LOPA]]\Bigr)u(t_k)}\COMMA\nonumber
\end{eqnarray}
where  $u=u(t)$ solves \VIZ{ODE}. Furthermore,  due to \VIZ{betterlie} and \VIZ{betterstrang} we have that
\begin{equation}\label{error_communication_lie_strang}
   \mathrm{TOL}  \sim \BIGO\Big(\frac{L^{d+1}}{q}\Big) \Delta t_{\mathrm{Lie}}\COMMA \quad  \mathrm{TOL} 
   \sim \BIGO(\frac{L^{2d+1}}{q}\Big) \Delta t_{\mathrm{Strang}}^2\PERIOD
\end{equation}
In the case of the randomized PCS the same reasoning as in Theorem~\ref{error:global} allows us
to iterate   the mean local error \VIZ{localmeanerror3} to obtain
\begin{align}\label{error_communication_random}
   \mathrm{TOL}  & \sim C_{\mathrm{Random}}(T) \Delta t_{\mathrm{Random}}\, \COMMA \nonumber \\
   \quad C_{\mathrm{Random}}(T) & =\max_{k=0,\dots,n}\EXPECT^\xi\left[ 4\LOPER^2_{\xi_1} +
    4\LOPER^2_{\xi_2} + 8\LOPER_{\xi_1}\LOPER_{\xi_2} -\LOPER^2\right]u(t_k)\COMMA 
\end{align}
where $u=u(t)$ solves \VIZ{ODE}. We now easily obtain that 
$$
  \EXPECT^\xi\left[ 4\LOPER^2_{\xi_1} + 4\LOPER^2_{\xi_2} + 8\LOPER_{\xi_1}\LOPER_{\xi_2} -\LOPER^2\right]u(t)
    =\left[4\LOPER^2_1+4\LOPER^2_2+\LOPER^2\right]u(t) \PERIOD
$$
Thus, due to the rigorous remainder bounds in Section~\ref{macroscopic} on the solution of \VIZ{ODE} such as Lemma~\ref{bound:L2_on_u},
we have that the term $\NORM{\left[4\LOPER^2_1+4\LOPER^2_2+\LOPER^2\right]u}$ is  of order $\BIGO(1)$ in the system size $N$, and we have 
\begin{equation}\label{error_communication_random_2}
   \mathrm{TOL}  \sim \BIGO(1) \Delta t_{\mathrm{Random}}\PERIOD
\end{equation}

In order to achieve the same  error tolerance TOL, \VIZ{error_communication_lie_strang} and
\VIZ{error_communication_random_2} imply the following relation between the respective time steps
\begin{equation} \label{order2}
  \delta t_{\mathrm{SSA}} \ll \Delta t_{\mathrm{Random}} \sim   \frac{L^{d+1}}{q}\Delta t_{\mathrm{Lie}}
             < \Delta t_{\mathrm{Lie}}\sim L^{d}\Delta t_{\mathrm{Strang}}^2 < \Delta t_{\mathrm{Strang}}\PERIOD
\end{equation}
Here $q$ is the diameter of each of the cells $C_k$ in Figure~\ref{fig:lattice_partition}, and $\delta t_{\mathrm{SSA}}=
\BIGO(1/N)$  is the stochastic time step (the waiting time) of the SSA algorithm \cite{Gillespie76}, which  is exponentially
distributed according to  \VIZ{totalrate}. 

\medskip
The relation \VIZ{order2} has several practical implications.
\begin{enumerate}[(i)]
\item The selection of the time window $\Delta t$ in each PCS is  intrinsically goal-oriented in the sense that it depends 
      directly on the macroscopic observable $f(\sigma)$ 
      through the commutator estimates of the solution to \VIZ{ODE}.
\item The random and deterministic PCS studied here are rigorously partially asynchronous as their
      respective time windows are much larger than  the SSA time step $\delta t_{\mathrm{SSA}}$
      for a given error  tolerance.
\item  The Lie scheme \VIZ{lie} is expected to parallelize better than the randomized  
       PCS  in  \cite{ShimAmar05b} when $L^{d+1}\ll q$, since it allows a $q$-times larger time step
       $\Delta t$ for the same accuracy. This outcome is also demonstrated in Figure~\ref{order}.
\item Finally, among the PCS we studied,  the Strang PCS yields  parallel schemes with 
      the least processor communication, at least when $L \sim \BIGO(1)$,  due to its higher order accuracy 
      and the commutator estimate \VIZ{betterstrang}.
\end{enumerate}
\medskip

\begin{example}\label{example:order}
{\rm
We demonstrate this comparison in a computational example in which a jump process defined by Arrhenius spin-flip dynamics
on a one-dimensional lattice was simulated. The simulated system corresponds to the Ising model with nearest-neighbor interactions
and spins taking values in $\{0,1\}$. The rate of the process is give by
$$
  c(x,\sigma) = c_d(1-\sigma(x)) + c_a\sigma(x)\EXP{-\beta U(x)}\COMMA
$$
where $U(x) = J (\sigma(x-1) + \sigma(x+1))+ \bar h$, and $c_d$, $c_a$, $\beta$, $J$, $h$ are the parameters of the model.

We verified the theoretical order of convergence by computing the error
$$
   \int_0^T |\EXPECT[C(t)] - \EXPECT[\tilde C(t)]|\, dt
$$
where $C(t)$ and $\tilde C(t)$ are the reference KMC and the FS-KMC solution, respectively, obtained by
averaging the spatial mean coverage process $C(t) = \sum_{x\in\LATT}\sigma_t(x)$ of the system 
over $K$ independent realizations. 
For the reference solution, the classical stochastic simulation algorithm (SSA) was used. 
In order to eliminate the impact of the statistical averaging error $K=10^5$ independent samples were used.
The error bars are below resolution of the graph depicted in Figure~\ref{order}.
In Figure~\ref{order} the error behavior is compared for different values of the splitting
time step $h \equiv \Delta t$ for the randomized PCS and the Lie splitting.
The lattice size is $N=800$ and the parameters of the system are $\beta=15$, $J=0.37$, $h=0.5$ and
$c_a=c_d=1$. For the fractional step algorithm four processors were used, thus the size of the sub-lattice is
$q=100$. The final time is chosen to be $T=4$. 
\begin{figure}[ht]
  \centerline{%
        \includegraphics[scale=0.4]{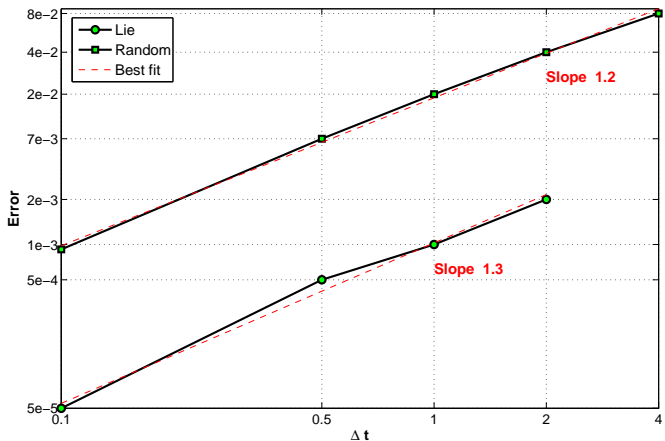}
  }%
  \caption{\label{order} Convergence of the weak error for deterministic and randomized Lie splitting.}
\end{figure}

}
\end{example}

\begin{example}
{\rm 
In this example we investigate the dependence of the weak error, as defined in the previous
example, on the sub-lattice parameter $q$. The model we used to run the simulation is Ising model,
as described in Example~\ref{example:order}. The parameters for the model are $\beta=5$, $J=1$,
$h=0.5$, and $c_a=c_d=1$. The final time is chosen to be $T=5$ and the dimension of the lattice
$N=480$. For the FS-KMC algorithm a constant, and rather large, time step parameter $\Delta t=5$
was used. For the FS-KMC algorithm we used $K=10^4$ samples to compute the mean value of the
solution on the interval $[0,T]$ and for the reference solution, which was obtained with the SSA
algorithm, $K=10^5$ samples were used. 
\begin{figure}[ht]
  \centerline{%
        \includegraphics[scale=0.5]{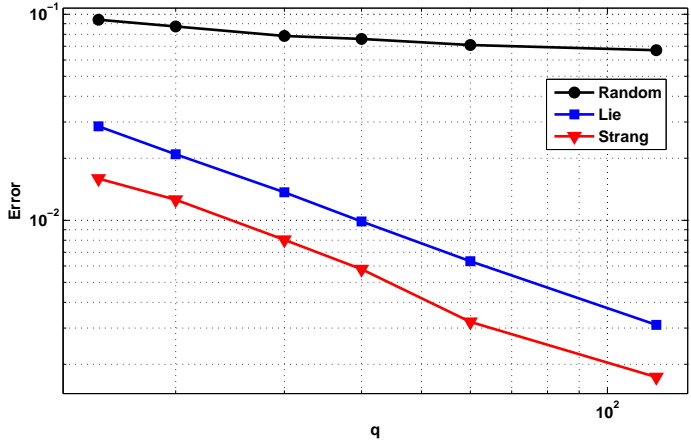}
  }%
  \caption{\label{fig:errorVSq} Dependence of the weak error on the sub-lattice size parameter $q$, see also \VIZ{order2}.}
\end{figure}
In Figure~\ref{fig:errorVSq} we can observe that the deterministic schedules of Lie and Strang give
better results than those of the random PCS. Also the Strang scheme has lower error than the Lie
scheme as expected from the theoretical analysis. Finally, the dependence of the error on
$\frac{1}{q}$ is also revealed, which in logarithmic scale is shown as a straight line. 
}
\end{example}

 \section{The infinite volume limit}\label{infinite_volume}

In this paper we considered interacting particle systems defined on a $d$-dimensional lattice
$\LATT$, as the numerical analysis and simulations for the parallel fractional step Kinetic Monte Carlo
are performed on a finite lattice of size $N$.
However, given the size of real
molecular systems it is  necessary that numerical estimates are independent of the system size $N$
as we showed in Section~\ref{macroscopic}.
Alternatively we can consider the case $N\to\infty$, e.g., by setting up our analysis on the infinite
lattice $\Lambda = \Z^d$. We outline the latter approach here for completeness of our analysis.
We refer to \cite{Liggett} for a comprehensive study of interacting particle systems set on infinite
lattices.

First, we consider the configuration space $\SIGMA=\SPINSP^{\Lambda}$, where $\Lambda = \Z^d$ and
the space of bounded continuous functions $\CBS$. Then the generator \VIZ{generator}
is defined on a suitable domain $\cal D(\LOPER)$,
$$
  \LOPER: D(\LOPER) \subset \CBS \mapsto \CBS \, .
$$
In this case, Theorem~\ref{error:global} is restated similarly to Theorem~3 in \cite{jahnke},
provided the solution $u=u(t)$ of \VIZ{ODE} satisfies 
$u(t_k) \in {\cal D}(\LOPA^{m_1}\LOPB^{m_2})$
for $|m|\le 3$ and $k=0,\dots,n$. As it was also pointed out in \cite{jahnke} this is in principle an
uncheckable hypothesis. However, this is not the case here:
due to  the results of Section~\ref{macroscopic} we have that if $f \in  \CSS{2}$ is a macroscopic
observable, where
\begin{equation*}
 \CSS{m} := \{ f\in\CBS \SEP  \sum_{k=1}^m\NORMA{k}{f} <\infty   \} \COMMA \forall m \in\mathbb{N}\COMMA
\end{equation*}
then $u(t) \in {\cal D}(\LOPA^{m_1}\LOPB^{m_2})$ due to  Theorem~\ref{main} and
Remark~\ref{bound:L2_on_u_easy}. Therefore, all estimates  of Section~\ref{macroscopic} hold also
true in the infinite lattice $\Lambda$, which is certainly not unexpected since all previous results
in $\LATT$ were independent of the system size $N$.

\section{Conclusions}

In this paper, we derived numerical error estimates for the Fractional Step Kinetic Monte Carlo (FS-KMC) algorithms 
proposed in \cite{AKPTX} for the parallel simulation
of interacting  particle systems on a lattice. These  algorithms have the capacity to simulate a wide range of spatio-temporal 
scales  of  spatially distributed, non-equilibrium physiochemical 
processes with complex chemistry and transport micro-mechanisms, while they can be tailored to specific hierarchical parallel 
architectures such as clusters of Graphical Processing Units.
A key aspect of 
our approach  relies on emphasizing a {\em goal-oriented} error analysis for macroscopic observables (e.g.,
density, energy, correlations, surface roughness), rather than 
focusing on strong topology estimates for individual trajectories or estimating probability distributions solving 
the Master Equation (Forward Kolmogorov Equation). 
Our analysis also addresses earlier work on parallel KMC algorithms \cite{ShimAmar05b, SPPARKS} 
that fit into the FS-KMC framework. Furthermore, moving  beyond  the parallelization problems discussed here, 
it appears that these  methodologies, introduced  in Section~\ref{macroscopic},   can be generally useful  in the development and study of
numerical approximations of molecular and other extended systems. 
Our error analysis allows us to address systematically the processor communication of
different parallelization strategies for KMC by  comparing their (partial) asynchrony, which in turn is measured by 
their respective  fractional step time-step for a prescribed  error tolerance.

\appendix

\begin{center}
\vspace*{0.2cm}
  {\bf APPENDIX}
\end{center}

\section{A general form of Gronwall's inequality}

For the sake of completeness we prove a variant of Gronwall's lemma for a particular case that appears in
the proof of Proposition~\ref{bound:derivative_of_u}. We prove it in the presence of two
equations, but the result can be easily generalized for a system of equations.

\begin{lemma} [Gronwall's inequality]\label{gronwall}
 Let $\vartheta$ and $\varphi$ satisfy the following inequalities
\begin{eqnarray*}
  \varphi(t) &\leq& \varphi(0) + \int_0^t \varphi(s)\, ds  \\ 
  \vartheta(t) &\leq& \vartheta(0) + \int_0^t\vartheta(s)\, ds + \int_0^t\varphi(s)\, ds
\end{eqnarray*}
then
\begin{eqnarray}
  \varphi(t)   &\leq& e^t\varphi(0) \\ 
  \vartheta(t) &\leq& e^t\vartheta(0) + (e^t+te^t-1)\varphi(0)
\end{eqnarray}
\end{lemma}

\begin{proof}
The first estimate follows directly from Gronwall's inequality. 
By integrating this inequality on $[0,t]$
$$
   \int_0^t \varphi(s) \leq (e^t -1)\varphi(0)\COMMA
$$
and by substituting this to the second inequality we obtain
$$
  \vartheta(t) \leq \vartheta(0) + (e^t-1)\varphi(0) + \int_0^t\vartheta(s)\, ds\PERIOD
$$
If we multiply by $e^{-t}$ and integrate on $[0,t]$ we have
$$
   \int_0^t \Bigl [e^{-r} \int_0^r \vartheta(r) \Bigr ]'\, dr \leq (1-e^{-t})\vartheta(0) +
    (e^t+te^t-1)\varphi(0)\COMMA
$$
and after straightforward calculations
$$
   \vartheta(t) \leq e^t \vartheta(0) + (e^t+te^t-1)\varphi(0)\PERIOD
$$
\end{proof}

\begin{remark}\label{remark:gronwall}
{\rm 
Let $\Phi(t)=(\varphi_1(t),\dots,\varphi_n(t))$ satisfying 
$$
\Phi(t) \leq \Phi(0) + \int_0^t A\Phi(s)\,ds
$$
where $A$ is a constant lower triangular matrix and the inequality has the meaning that it is true
component-wise, then
$$
  \Phi(t) \leq B(t)\Phi(0)\COMMA
$$
where $B$ is a lower triangular matrix with elements exponentially depending on $t$.
}
\end{remark}


\begin{thebibliography}{10}

\bibitem{AKPTX}
{\sc G.~{Arampatzis}, M.~A. {Katsoulakis}, P.~{Plech\'a\v{c}}, M.~{Taufer}, and
  L.~{Xu}}, {\em {Hierarchical fractional-step approximations and parallel
  kinetic {M}onte {C}arlo algorithms}}, J. Comp. Phys.,  (2012), {\tt doi: 10.1016/j.jcp.2012.07.017}.

\bibitem{Auerbach}
{\sc S.~M. Auerbach}, {\em Theory and simulation of jump dynamics, diffusion
  and phase equilibrium in nanopores.}, Int. Rev. Phys. Chem., 19 (2000).

\bibitem{BKL75}
{\sc A.~B. Bortz, M.~H. Kalos, and J.~L. Lebowitz}, {\em A new algorithm for
  {M}onte {C}arlo simulation of {I}sing spin systems}, J. Comp. Phys., 17
  (1975), pp.~10--18.

\bibitem{ACDV}
{\sc A.~Chatterjee and D.~G. Vlachos}, {\em {An overview of spatial microscopic
  and accelerated kinetic Monte Carlo methods}}, {J. Comput.-Aided Mater.
  Design}, {14} ({2007}), pp.~{253--308}.

\bibitem{Lubachevsky93}
{\sc S.~G. Eick, A.~G. Greenberg, B.~D. Lubachevsky, and A.~Weiss}, {\em
  Synchronous relaxation for parallel simulations with applications to
  circuit-switched networks}, ACM Trans. Model. Comput. Simul., 3 (1993),
  pp.~287--314.

\bibitem{Gardiner04}
{\sc Crispin Gardiner}, {\em Handbook of Stochastic Methods: for Physics,
  Chemistry and the Natural Sciences}, Springer, 4th~ed., 2009.

\bibitem{Gillespie76}
{\sc D.~T. Gillespie}, {\em A general method for numerically simulating the
  stochastic time evolution of coupled chemical reactions}, J. of Comp. Phys.,
  22 (1976), pp.~403--434.

\bibitem{Hairer}
{\sc Ernst Hairer, Christian Lubich, and Gerhard Wanner}, {\em Geometric
  numerical integration}, vol.~31 of Springer Series in Computational
  Mathematics, Springer-Verlag, Berlin, second~ed., 2006.
\newblock Structure-preserving algorithms for ordinary differential equations.

\bibitem{Nicol}
{\sc P.~Heidelberger and D.~M. Nicol}, {\em Conservative parallel simulation of
  continuous time markov chains using uniformization}, IEEE Trans. Parallel
  Distrib. Syst., 4 (1993), pp.~906--921.

\bibitem{jahnke}
{\sc T.~Jahnke and D.~Alt{\i}ntan}, {\em Efficient simulation of discrete
  stochastic reaction systems with a splitting method}, BIT, 50 (2010),
  pp.~797--822.

\bibitem{KL}
{\sc C.~Kipnis and C.~Landim}, {\em Scaling Limits of Interacting Particle
  Systems}, Springer-Verlag, 1999.

\bibitem{Korniss99}
{\sc G.~Korniss, M.~A. Novotny, and P.~A. Rikvold}, {\em Parallelization of a
  dynamic {M}onte {C}arlo algorithm: {A} partially rejection-free conservative
  approach}, J. Comp. Phys., 153 (1999), pp.~488--508.

\bibitem{Kurtz}
{\sc T.~G. Kurtz}, {\em A random {T}rotter product formula}, Proc. Amer. Math.
  Soc., 35 (1972), pp.~147--154.

\bibitem{binder}
{\sc D.~P. Landau and K.~Binder}, {\em A Guide to {M}onte {C}arlo Simulations
  in Statistical Physics}, Cambridge University Press, Cambridge, 2000.

\bibitem{Liggett}
{\sc Thomas~M. Liggett}, {\em Interacting Particle Systems}, vol.~276 of
  Grundlehren der mathematischen {W}issenschaften, Springer-Verlag, New York,
  Berlin, Heidelberg, Tokyo, 1985.

\bibitem{evans09}
{\sc Da-Jiang Liu and J.~W. Evans}, {\em {Atomistic and multiscale modeling of
  CO-oxidation on Pd(100) and Rh(100): From nanoscale fluctuations to mesoscale
  reaction fronts}}, {Surf. Science}, {603} ({2009}), pp.~{1706--1716}.

\bibitem{Lubachevsky88}
{\sc B.~D. Lubachevsky}, {\em Efficient parallel simulations of dynamic {I}sing
  spin systems}, J. Comput. Phys., 75 (1988), pp.~103--122.

\bibitem{MerickFichthorn07}
{\sc M.~Merrick and K.~A. Fichthorn}, {\em Synchronous relaxation algorithm for
  parallel kinetic {M}onte {C}arlo simulations of thin film growth}, Phys. Rev.
  E, 75 (2007), p.~011606.

\bibitem{ShimAmar09}
{\sc G.~Nandipati, Y.~Shim, J.~G. Amar, A.~Karim, A.~Kara, T.~S. Rahman, and
  O.~Trushin}, {\em Parallel kinetic {M}onte {C}arlo simulations of {A}g(111)
  island coarsening using a large database}, Journal of Physics Condensed
  Matter, 21 (2009), p.~084214.

\bibitem{SPPARKS}
{\sc S.~Plimpton, C.~Battaile, M.~Chandross, L.~Holm, A.~Thompson, V.~Tikare,
  G.~Wagner, E.~Webb, X.~Zhou, C.~Garcia Cardona, and A.~Slepoy}, {\em
  {Crossing the Mesoscale No-Man's Land via Parallel Kinetic Monte Carlo}},
  Tech. Report SAND2009-6226, Sandia National Laboratory, 2009.

\bibitem{Reuter1}
{\sc K~Reuter, D~Frenkel, and M~Scheffler}, {\em {The steady state of
  heterogeneous catalysis, studied by first-principles statistical mechanics}},
  {Physical Review Letters}, {93} ({2004}).

\bibitem{ShimAmar05}
{\sc Y.~Shim and J.~G. Amar}, {\em Rigorous synchronous relaxation algorithm
  for parallel kinetic {M}onte {C}arlo simulations of thin film growth}, Phys.
  Rev. B, 71 (2005), p.~115436.

\bibitem{ShimAmar05b}
\leavevmode\vrule height 2pt depth -1.6pt width 23pt, {\em Semirigorous
  synchronous relaxation algorithm for parallel kinetic {M}onte {C}arlo
  simulations of thin film growth}, Phys. Rev. B, 71 (2005), p.~125432.

\bibitem{Szabo}
{\sc G.~Szabo and G.~Fath}, {\em {Evolutionary games on graphs}}, {Physics
  Reports}, {446} ({2007}), pp.~{97--216}.

\bibitem{Trotter}
{\sc H.~F. Trotter}, {\em On the product of semi-groups of operators}, Proc.
  Amer. Math. Soc., 10 (1959), pp.~545--551.

\end{thebibliography}
\end{document}